\documentclass[11pt,reqno]{amsart}
\usepackage{latexsym,amsmath,amsfonts,amscd,amssymb}
\usepackage{graphics}
\usepackage[all]{xy}
\usepackage{slashed}
\usepackage{amsthm}
\usepackage{stmaryrd}
\usepackage{tabu}
\usepackage{longtable}
\usepackage{color}
\definecolor{cobalt}{RGB}{61,89,171}
\usepackage[colorlinks,citecolor=cobalt,linkcolor=cobalt,urlcolor=cobalt,pdfpagemode=UseNone,backref = page]{hyperref}
\usepackage{makecell}

\usepackage{adjustbox}

\newcommand{\Spin}{\operatorname{Spin}}

\newcommand{\SU}{\operatorname{SU}}
\newcommand{\UU}{\operatorname{U}}
\newcommand{\SO}{\operatorname{SO}}
\newcommand{\Sp}{\operatorname{Sp}}

\newcommand{\tr}{\operatorname{tr}}

\newcommand{\im}{\operatorname{im}}

\newcommand{\Stab}{\operatorname{Stab}}

\newcommand{\Cl}{\operatorname{Cl}}

\newcommand{\sD}{\slashed{D}}

\DeclareMathOperator{\Sym}{Sym}

\DeclareMathOperator{\End}{End}

\DeclareMathOperator{\pr}{pr}

\DeclareMathOperator{\PSpin}{P_{\Spin}}
\DeclareMathOperator{\PSO}{P_{\SO}}

\newcommand{\la}{\langle}
\newcommand{\ra}{\rangle}

\newcommand{\Ad}{\operatorname{Ad}}

\newcommand{\onabla}{\overline{\nabla}}

\newcommand{\bk}{{\mathbf{k}}}

\newcommand{\cDD}{{\mathcal D}}

\newcommand{\cN}{{\mathcal N}}

\newcommand{\fX}{{\mathfrak X}}

\newcommand{\CC}{{\mathbb C}}

\newcommand{\HH}{{\mathbb H}}

\newcommand{\RR}{{\mathbb R}}

\newcommand{\ClC}{{\CC\mathrm{l}}}

\renewcommand{\a}{\alpha}
\renewcommand{\b}{\beta}

\newcommand{\g}{\gamma}
\newcommand{\e}{\varepsilon}

\renewcommand{\l}{\lambda}
\renewcommand{\k}{\kappa}
\newcommand{\m}{\mu}
\newcommand{\n}{\nu}
\renewcommand{\o}{\omega}
\newcommand{\p}{\phi}

\newcommand{\vp}{\varphi}

\newcommand{\s}{\sigma}
\renewcommand{\t}{\tau}

\newcommand{\G}{\Gamma}
\renewcommand{\O}{\Omega}
\renewcommand{\S}{\Sigma}
\newcommand{\D}{\Delta}
\renewcommand{\L}{\Lambda}

\newcommand{\frX}{{\mathfrak{X}}}

\newcommand{\frspin}{{\mathfrak{spin}}}

\newcommand{\frg}{{\mathfrak{g}}}

\newcommand{\frn}{{\mathfrak{n}}}

\newcommand{\frsu}{{\mathfrak{su}}}
\newcommand{\imag}{\mathbf{i}}
\newcommand{\jota}{{\mathrm{j}}}

\newcommand{\rI}{\mathrm{I}}
\newcommand{\rL}{\mathrm{L}}

\newcommand{\rA}{\mathrm{A}}
\newcommand{\rG}{\mathrm{G}}

\newcommand{\rN}{\mathrm{N}}

\newtheorem{theorem}{Theorem}[section]
\newtheorem{proposition}[theorem]{Proposition}
\newtheorem{lemma}[theorem]{Lemma}
\newtheorem{corollary}[theorem]{Corollary}

\theoremstyle{definition}
\newtheorem{definition}[theorem]{Definition}

\theoremstyle{remark}

\newtheorem{remark}[theorem]{Remark}

\title{Spin-harmonic structures and nilmanifolds}

\author[G. Bazzoni]{Giovanni Bazzoni}
\address{Dipartimento di Scienza ed Alta Tecnologia, Universit\`a degli Studi dell'Insubria, Via Valleggio 11, 22100 Como, Italy}
\email{giovanni.bazzoni@uninsubria.it}

\author[L. Mart\'{\i}n-Merch\'an]{Luc\'{\i}a Mart\'{\i}n-Merch\'an}
\address{Dipartimento di Matematica “G. Peano”, Università degli Studi di Torino, Via Carlo Alberto 10, 10123 Torino, Italy}
\email{lucia.martinmerchan@unito.it}

\author[V. Mu\~{n}oz]{Vicente Mu\~{n}oz}
\address{Departamento de Algebra, Geometr\'{\i}a y Topolog\'{\i}a, 
Universidad de M\'alaga, Campus de Teatinos, 29071 M\'alaga, Spain}
\email{vicente.munoz@uma.es}


\begin{document}

\begin{abstract}
We introduce spin-harmonic structures, a class of geometric structures on Riemannian manifolds of low dimension which are defined by a harmonic unit-length spinor. Such structures are related to $\SU(2)$ ($\dim=4,5$), $\SU(3)$ ($\dim=6$) and $\rG_2$ ($\dim=7$) structures; in dimension 8, a spin-harmonic structure is equivalent to a balanced $\Spin(7)$ structure. As an application, we obtain examples of compact 8-manifolds endowed with non-integrable $\Spin(7)$ structures of balanced type.
\end{abstract}

\maketitle


\section{Introduction}\label{sec:1}

In 1980 Thomas Friedrich proved a remarkable inequality involving the scalar curvature of a compact, spin Riemannian manifold and the first eigenvalue of the Dirac operator, see \cite{Friedrich}. This triggered a deep analysis of spin Riemannian manifolds; particular emphasis was put on which compact manifolds admitted parallel, twistor or Killing spinors, see for instance \cite{Baer,BFGK,Moroianu-Semmelmann}. In particular, it was soon clarified that Riemannian manifolds endowed with a parallel spinor are related to Riemannian manifolds with {\em special} holonomy, i.e.~Riemannian manifolds whose Riemannian holonomy is contained in $\SU(n)$, $\Sp(n)$, $\rG_2$, or $\Spin(7)$; notice that the Ricci curvature of a compact Riemannian manifold endowed with a parallel spinor vanishes.

Relaxing the requirement to have a parallel spinor, it was later shown that many non-integrable $\rG$ structures, $\rG\subset\SO(n)$ being a closed subgroup, can be understood in terms of nowhere vanishing spinors, generalizing the case of parallel spinors. 
For instance, in \cite{ACFH15} the authors described $\SU(3)$ and $\rG_2$ structures in dimensions 6 and 7 respectively using a unit-length spinor. Not only does the spinorial approach offer an alternative frame for telling apart different classes of such structures, but also provides a unifying language showing how the same spinor is responsible for the emerging of both structures. 

$\SU(2)$ structures in dimension 5 have been introduced by Conti and Salamon in \cite{CS07} and classified by Bedulli and Vezzoni in \cite{BV09} in terms of the exterior derivatives of the corresponding defining forms -- see Section \ref{sec:su(2)}. In \cite{CS07}, the study of $\SU(2)$ structures in dimension 5 was certainly motivated by spinors, concretely, generalized Killing spinors. However, no spinorial description of such structures is available; the first goal of this paper is to tackle this question. We do this in Section \ref{sec:su(2)}.

As for $\Spin(7)$ structures on 8-dimensional manifolds, they can be described in terms of a triple cross product on each tangent space; an equivalent description can be given in terms of the so-called fundamental 4-form $\Omega$. The different types of $\Spin(7)$ structures were classified by Fern\'andez in \cite{MF86} using the triple cross product: there exist two pure classes, called {\em balanced} and {\em locally conformally parallel}. An equivalent classification is obtained by considering the fundamental form: balanced $\Spin(7)$ structures are characterized by the equation $\star(d\Omega)\wedge\Omega=0$, while the 4-form of a locally conformally parallel $\Spin(7)$ structure satisfies $d\Omega=\Omega\wedge\theta$ for a closed 1-form $\theta$, called the {\em Lee form}. In \cite{Ivanov} Ivanov discovered that the unit-length spinor which characterizes balanced $\Spin(7)$ structures is {\em harmonic}, that is, it lies in the kernel of the Dirac operator $\sD$, but gave no further application of this fact. Notice that Hitchin proved in \cite{Hitchin2} than every compact spin 8-manifold carries a harmonic spinor; not much is known, however, about zeroes of harmonic spinors (see \cite{Baer2}).

A systematic spinorial approach to $\Spin(7)$, along the lines of \cite{ACFH15}, was taken by the second author in \cite{MM18}. In particular, the observation that balanced $\Spin(7)$ structures are equivalent to unit-length harmonic spinors was exploited in \cite{MM18} to construct examples of balanced $\Spin(7)$ structures on 8-dimensional nilmanifolds and solvmanifolds. There it became clear that the spinorial approach has some practical advantages on the ``classical'' one, which uses the 4-form. The principle we follow in this paper is that albeit both the equation $\sD\eta=0$ for a unit-length spinor and the equation $\star(d\Omega)\wedge\Omega=0$ for a 4-form are non-linear, the first one seems to be more tractable, at least if one is interested in constructing examples of balanced $\Spin(7)$ structures on compact quotients of simply connected nilpotent and solvable Lie groups, that is, on nilmanifolds and solvmanifolds.

Indeed, the second goal of this paper is to construct examples of balanced $\Spin(7)$ structures on 8-dimensional nilmanifolds. The first known example of such a structure is a nilmanifold described by Fern\'andez in \cite{MF87}. Further examples are discussed in \cite{MC95,Ivanov}. Notice that the classification of 8-dimensional nilpotent Lie algebras is not known. Even if it were, however, it is not immediately clear how to sift through them in order to find those admitting balanced $\Spin(7)$ structures (for instance, the balanced condition is not of cohomological type).

We describe briefly the idea behind the construction. As we pointed out, it is very natural to consider $\Spin(7)$ structures in dimension 8 defined by a chiral, unit-length, and harmonic spinor. Nothing hinders, however, to consider $\rG_2$, $\SU(3)$ and $\SU(2)$ structures in dimensions 7, 6 and 5 respectively, such that the defining spinor is harmonic. Using the spinorial approach of \cite{ACFH15}, one can precisely track which classes of $\rG_2$ and $\SU(3)$ are defined by harmonic spinors; moreover, our spinorial description also allows to pinpoint which classes of $\SU(2)$ structures arise from a harmonic spinor. While $\Spin(7)$ structures defined by a harmonic spinor form a pure class, the same is not true in lower dimensions; for instance, in dimension 5, the requirement to be harmonic for the corresponding spinor turns out to be quite loose.

Viceversa, beginning with an $\SU(2)$ structure on a 5-manifold (resp.\ an $\SU(3)$ structure on a 6-manifold, or a $\rG_2$ structure on a 7-manifold), defined by a unit-length harmonic spinor, one can multiply by a flat torus $T^k$, $k=3,2,1$, to obtain a $\Spin(7)$ structure in dimension 8 defined by a harmonic spinor, that is, a balanced structure.

In order to construct such examples, we need a formula for the Dirac operator acting on a particular class of spinors on a nilmanifold $\Gamma\backslash\rG$; namely, we restrict to {\em left-invariant} spinors, those which come from left-invariant spinors on the Lie group $\rG$; for more details, we refer the reader to Section \ref{sec:Dirac-Operator}. The following formula is obtained in Proposition \ref{invariant-Dirac-operator} and expresses the Dirac operator on invariant spinors in a purely algebraic way:
\[
4 \sD \p  = -\sum_{i=1}^n{(e^i \wedge de^i + i(e_i)de^i)}\p \,. 
\]

In Section \ref{sec:harm-nilm} we rely on the existing classification of nilpotent Lie algebras up to dimension 6 (see for instance \cite{BM2}) for solving the equation $\sD\eta=0$ in the space of left-invariant spinors on low dimensional nilmanifolds. In particular, we show which metric nilpotent Lie algebras in dimensions 4, 5, and 6 admit a harmonic spinor -- see Theorems \ref{4-harmonic}, \ref{5-harmonic} and \ref{6-harmonic-decomposable}, and Subsection \ref{6-harmonic-non-decomposable}. We point out here that, although the proof is achieved by a case-by-case analysis, ours is the first systematic spinorial approach to the study of geometric structures on nilmanifolds.




This paper is organized as follows: in Section \ref{sec:Preliminaries} we review the necessary preliminaries on Clifford algebras and spinor bundles. Section \ref{Spinors} reviews the spinorial description of $\Spin(7)$, $\rG_2$ and $\SU(3)$ structures; we introduce the notion of {\em spin-harmonic} geometric structure, that is, a geometric structure defined by a harmonic unit-length spinor. In Section \ref{sec:su(2)} we carry out the spinorial classification of $\SU(2)$ structures on 5-manifolds. In Section \ref{sec:Dirac-Operator} we consider left-invariant spinors on simply connected Lie groups, finding a general formula for the Dirac operator -- see Proposition \ref{invariant-Dirac-operator} --  which we specialize to the case of nilpotent and (a certain kind of) solvable Lie groups. Using this formula, in Section \ref{sec:harm-nilm} we tackle nilpotent Lie algebras (and nilmanifolds) in dimensions 4, 5, and 6. In dimension 4, a non-abelian nilpotent Lie algebra admits no metric with harmonic spinors. In dimension 5 we classify metric nilpotent Lie algebras and determine those which admit harmonic spinors. Finally, in dimension 6, either we provide a metric on the Lie algebra which admits harmonic spinors, or we show that no such metric exists.

\noindent\textbf{Acknowledgements.} We are grateful to Anna Fino for useful conversations. We also thank Spiro Karigiannis for reading the manuscript carefully and checking the computations. The authors were partially supported by Project MINECO (Spain) MTM2015-63612-P. The first author was supported by a {\em Juan de la Cierva - Incorporaci\'on} Fellowship of Spanish Ministerio de Ciencia, Innovaci\'on y Universidades. The second author acknowledges financial support by an FPU Grant (FPU16/03475).


\section{Preliminaries}\label{sec:Preliminaries}

In this section we recall some basic aspects about the representation theory of Clifford algebras, in the real and the complex case, as well as generalities on spinor bundles; further details can be found in \cite{F00} and \cite{LM89}. 


\subsection{Representations of the real Clifford algebra}\label{Cl-representations}

If $n\not\equiv 3 \pmod{4}$, the real Clifford algebra $\Cl_n$ of $\left(\RR^n,\sum_{j=1}^n x_j^2\right)$ is isomorphic to the algebra of $l$-dimensional matrices with coefficients in the (skew) field $\bk$, $\bk \in \{ \RR, \CC, \HH \}$; we denote this algebra by $\bk(l)$. If $n \equiv 3 \pmod{4}$, $\Cl_n$ is isomorphic to $\bk(l)\oplus \bk(l)$. In low dimensions, the following isomorphisms hold (see \cite[Chapter 1, Theorem 4.3]{LM89}):

\begin{minipage}[t]{0.48\textwidth}
\begin{itemize}
 \item $\Cl_1=\CC$;
 \item $\Cl_2=\HH$;
 \item $\Cl_3=\HH \oplus \HH$;
 \item $\Cl_4=\HH(2)$;
\end{itemize}
\end{minipage}
\begin{minipage}[t]{0.48\textwidth}
\begin{itemize}
 \item $\Cl_5=\CC(4)$;
 \item $\Cl_6=\RR(8)$;
 \item $\Cl_7=\RR(8)\oplus \RR(8)$;
 \item $\Cl_{8}=\RR(16)$.
\end{itemize}
\end{minipage}

Isomorphisms in higher dimensions are determined by the property $\Cl_{n+8}=\Cl_{n}\otimes \Cl_8$. As a consequence, there is a unique equivalence class of irreducible representations of $\Cl_n$ if $n \not\equiv 3 \pmod{4}$ and two different ones if $n\equiv 3 \pmod{4}$; these are determined by the image of the volume form, which can be $\rI$ or $-\rI$ \cite[Chapter 1, Proposition 5.9]{LM89}. 

By construction, the even part of the Clifford algebra $\Cl_n$, denoted $\Cl_n^0$, is isomorphic to the Clifford algebra $\Cl_{n-1}$; using this, one can construct irreducible representations of $\Cl_{n-1}$ from irreducible representations of $\Cl_n$ by using the following result, which is essentially a reformulation of \cite[Chapter 1, Proposition 5.12]{LM89}.

\begin{proposition} \label{restriction}
Let $W$ be a $\bk$-vector space and let $\rho_n \colon \Cl_n \to \End_{\bk}(W)$ be an irreducible representation. Write $\RR^n=\RR^{n-1}\oplus\RR$, where the second factor is generated by a unit-length vector $e_n$, and denote by $i_{n-1} \colon \Cl_{n-1}\to\Cl^0_n$ the extension to $\Cl_{n-1}$ of the map $\RR^{n-1}\to\Cl^0_n$, $v\mapsto ve_n$; define $\rho_{n-1}= \rho_n \circ i_{n-1} \colon \Cl_{n-1} \to\End_\bk(W)$. Then,
\begin{enumerate}
\item If $n\equiv 0\pmod{4}$ the representation $\rho_{n-1}$ splits into two irreducible and inequivalent representations, $\rho_{n-1}^\pm$. These are the eigenspaces $W^{\pm}$ of the endomorphism,
$\rho_n(\nu_n)\colon W \to W$, where $\nu_n$ is the volume form in $\RR^n$.

\item If $n\equiv 1,2\pmod{8}$, the representation $\rho_{n-1}$ splits into two irreducible equivalent representations.

\item If $n\equiv 3,5,6,7 \pmod{8}$, the representation $\rho_{n-1}$ is irreducible.
\end{enumerate}
\end{proposition}

In this paper, we will work with the following $6$-dimensional real representation of $\Cl_6$:
\vskip -0.5 cm
\noindent\begin{minipage}[t]{0.48\textwidth}
\begin{align*}
e_1 = &+ E_{18} + E_{27} - E_{36} - E_{45}, \\
e_2 = & -E_{17} + E_{28} + E_{35} - E_{46}, \\
e_3 = & -E_{16} + E_{25} - E_{38} + E_{47},
\end{align*}
\end{minipage}
\noindent\begin{minipage}[t]{0.48\textwidth}
\begin{align*}
e_4 = & -E_{15} - E_{26} - E_{37} - E_{48}, \\
e_5 = & -E_{13} - E_{24} + E_{57} + E_{68}, \\
e_6 = & + E_{14} - E_{23} - E_{58} + E_{67},
\end{align*}
\end{minipage}
%
%

where the matrices $E_{ij}$ denotes the skew-symmetric endomorphism of $\RR^8$ that maps the $i^{th}$ vector of the canonical base to the $j^{th}$ one and is zero on the orthogonal complement. 

\subsection{Representations of the complex Clifford algebra}\label{ClC-representations}

Let $\ClC_n$ be the complex Clifford algebra of $\left(\CC^n, \sum_{j=1}^n{z_j^2}\right)$. A construction of an irreducible representation of $\ClC_n$ can be found in \cite{F00}. There exist a $2^k$-dimensional complex vector space $\D_{2k}$ and isomorphisms
\begin{align*}
 \k_{2k} \colon & \ClC_{2k} \to \End_\CC(\D_{2k}), \\
 \tilde{\k}_{2k+1} \colon & \ClC_{2k+1} \to \End_\CC(\D_{2k})\oplus \End_\CC(\D_{2k})\,.
\end{align*}
Let $\pr_1\colon \End_\CC(\D_{2k})\oplus \End_\CC(\D_{2k}) \to \End_\CC(\D_{2k})$ be the projection onto the first summand. The complex representation of $\ClC_n$ is defined as $\k_{n}$ if $n=2k$ or $\pr_1\circ \tilde{\k}_n$ if $n=2k+1$. 

Then $\D_{2k}$ is irreducible as a representation of $\ClC_n$ and is used to define the complex spin representation: this is the restriction of $\k_n$ to $\Spin(n)\subset\ClC^0_n$. This representation is faithful and irreducible if $n=2k+1$; however, if $n=2k$, it splits into two irreducible summands $\D_{2k}^\pm$, which are the eigenspaces of eigenvalue $\pm 1$ of the $\Spin(n)$-equivariant endomorphism $\k_n(\nu_n^\CC)$, where $\nu_n^\CC=\imag^k\nu_n$.

Depending on the dimension, the complex vector space $\D_{2k}$ is endowed with a real structure $\varphi$ or a quaternionic structure $\jota_2$. These are antilinear endomorphisms of $\D_{2k}$ such that $\vp^2=\rI$ and $\jota_2^2=-\rI$; they commute or anticommute with the Clifford product, determining a real or quaternionic representation of $\Spin(n)$. The precise result is contained in the following proposition (see \cite[Chapter 1]{F00}):

\begin{proposition}\label{real-and-quaternionic-structures}
Suppose $n=2k+r$, with $r\in\{0,1\}$. 
\begin{enumerate}
\item If $k\equiv 0,3\pmod{4}$, then $\D_{2k}$ has a real structure $\vp$ with  $\vp \circ \k_n(v)=(-1)^{k+1} \k_n(v) \circ \vp$ for any $v\in \RR^{n}$.
\item If $k\equiv 1,2 \pmod{4}$, then $\D_{2k}$ has a quaternionic structure $\jota_2$ with  $\jota_2 \circ \k_n(v)=(-1)^{k+1} \k_n(v) \circ \jota_2$ for any $v\in \RR^{n}$.
\end{enumerate}
For the cases in which $\D_{2k}$ is decomposable as a $\Spin(n)$ representation, one has

\begin{minipage}[t]{0.48\textwidth}
\begin{itemize}
 \item $\vp(\D_{8p}^{\pm}) =  \D_{8p}^{\pm}$;
 \item $\vp(\D_{8p+6}^\pm) =  \D_{8p+6}^\mp$;
\end{itemize}
\end{minipage}
\begin{minipage}[t]{0.48\textwidth}
\begin{itemize}
 \item $\jota_2(\D_{8p+2}^\pm) = \D_{8p+2}^\mp$;
 \item $\jota_2(\D_{8p+4}^{\pm})=  \D_{8p+4}^\pm$.
\end{itemize}
\end{minipage}
\end{proposition}
We denote also by $(\D_{8p}^+)_{\pm}$, $(\D_{8p}^-)_{\pm}$ and $(\D_{8p+6})_\pm$ the eigenspaces of eigenvalue $\pm 1$ of $\vp$ on $\D_{8p}^+$, $\D_{8p}^-$ and $(\D_{8p+6})_\pm$ respectively. If $n=8p+q$ with $0\leq q \leq 7$ then $\Cl_n$ is isomorphic via $\tilde{\k}_n$ if $k\equiv 3\pmod{4}$, or via $\k_n$ otherwise, to:

\begin{minipage}[t]{0.44\textwidth}
\begin{itemize}
 \item[$q=0$:] $\End_\RR((\D_{8p}^+)_+\oplus (\D_{8p}^-)_-)$,
 \item[$q=1$:] $\End_\CC(\D_{8p})$, 
 \item[$q=2$:] $\End_\HH(\D_{8p+2})$,
 \item[$q=3$:] $\End_\HH(\D_{8p+2})\oplus \End_\HH(\D_{8p+2})$,
\end{itemize}
\end{minipage}
\begin{minipage}[t]{0.52\textwidth}
\begin{itemize}
 \item[$q=4$:] $\End_\HH(\D_{8p+4})$,
 \item[$q=5$:] $\End_\CC(\D_{8p+4})$,
 \item[$q=6$:] $\End_\RR((\D_{8p+6})_+)$,
 \item[$q=7$:] $\End_\RR((\D_{8p+6})_+)\oplus \End_\RR((\D_{8p+6})_+)$.
\end{itemize}
\end{minipage} 

\begin{remark}
If $n\equiv 2,3 \pmod{8}$ then $\jota_2$ is a quaternionic structure that commutes with the Clifford product and if $n\equiv 4 \pmod{8}$ then $\nu_4 \jota_2$ has the same property. That explains the notations $\End_\HH(\D_{8p+2})$ and $\End_\HH(\Delta_{8p+4})$.
\end{remark}

In addition, the representation $\D_{2k}$ is equipped with a hermitian product $h$ that makes the Clifford product by vectors on $\RR^{2k}$ and $\RR^{2k+1}$ a skew-symmetric endomorphism. We construct from it a scalar product on the irreducible representation of the Clifford algebra using standard results of real and quaternionic structures on irreducible representations applied to the $\Spin(2k+1)$ module $\D_{2k}$.
\begin{enumerate}
\item  If $k\equiv 0,3\pmod{4}$ the restriction of $h$ to $(\D_{2k})_{\pm}$ is real valued. Moreover, the spaces $\Delta_{2k}^\pm$ are orthogonal if $k\equiv 0 \pmod{4}$ because the multiplication by $\n_n^\CC$ is unitary.
\item If $k\equiv 1,2 \pmod{4}$ then $h(\jota_2\p,\jota_2\eta)=\overline{h(\p,\eta)}$, hence $\jota_2$ is an isometry for the real part of $h$.
\end{enumerate}
In both cases, we denote by $\la\cdot,\cdot\ra$ the real part of $h$.

\subsection{Spinor bundles}\label{Spinorial}

Let $(M,g)$ be an oriented $n$-dimensional spin manifold and let $\Ad\colon \PSpin(M) \to \PSO(M)$ be a spin structure. Let $W$ be a $\bk$ vector space and $\rho_n \colon \Cl_n \to \End_{\bk}(W)$ an irreducible representation. Recall that for $n\equiv 0 \pmod{4}$ there is a splitting $W=W^+\oplus W^{-}$ into $\Spin(n)$ irreducible representations (see Proposition \ref{restriction}).
\begin{definition}
A {\em real spinor bundle} over $M$ is $\S(M)=\PSpin(M) \times_{\rho_n} W$, for an irreducible representation $\rho_n \colon \Cl_n \to \End_\bk(W)$. If $n\equiv 0 \pmod{4}$, the {\em positive} and {\em negative subbundles} are $\S^{\pm}(M)= \PSpin(M) \times_{\rho_n} W^{\pm}$.
\end{definition}

Let $\Cl(M)$ denote the bundle whose fiber over $p\in M$ is the Clifford algebra of $(T_pM,g_p)$; the spinor bundle is a $\Cl(M)$-module with the Clifford product by a vector field $X\in\fX(M)$ given by 
\[
X [\tilde{F},v]=\left[\tilde{F}, \sum_i{ X^i\rho_n(e_i)v}\right]\,;
\]
here $X^i$ are the coordinates of $X$ with respect to the orthonormal frame $F=\Ad(\tilde{F})$. The Clifford multiplication extends to $\L^k T^*M$ in the following way:
\begin{itemize}
 \item the product with a covector is defined by $X^*\phi = X\phi$, with canonical identification between the tangent and the cotangent bundle given by the metric: $X^*=g(X,\cdot)$.
\item If the product is defined on $\L^l T^*M$ when $l\leq k$, we define 
\[(X^*\wedge \beta) \phi = X (\beta \phi) + (i(X)\beta) \phi,\]
where $i(X)\beta $ denotes the contraction, $\beta \in \L^kT^*M$ and $X\in\fX(M)$. This product is extended linearly to $\L ^{k+1}T^*M$. 
\end{itemize}

The relation among representations of $\Cl_n$ determine relations among spinor bundles. For instance, we have the following result:  

\begin{lemma} \label{Spinor-Bundle-Relation}
Let $(M,g)$ be an $n$-dimensional spin manifold with $n=8p+8-m$ and $4\leq m<8$. Consider the Riemannian manifold $(M\times \RR^m, g + g_m)$, where $g_m$ is the canonical metric on $\RR^m$ with orthonormal basis $(e_{n+1},\ldots,e_{8p+8})$. Denote by $\pr_1\colon M \times \RR^m \to M$ the canonical projection.
\begin{enumerate} 
\item There is a bijection between spin structures on $M$ and spin structures on $M\times \RR^m$. 
\item \label{S-B-R-II} The spinor bundles are related by $\S^+(M\times \RR^m)= \pr_1^*\S(M)$ with Clifford product $X(\phi,t)=(X e_{n+1} \p,t)$ for $X\in\fX(M)$.
\end{enumerate}
\end{lemma}
\begin{proof}
Denote by $i\colon M\hookrightarrow M \times \RR^m$ the canonical inclusion. First of all, $\PSO (M\times \RR^m)=\pr_1^* \PSO (M)$. Therefore, each spin structure on $M$ determines a spin structure on $M\times \RR^m$ by $\PSpin (M\times \RR^m)=\pr_1^* \PSpin(M)\times_{\Spin(n)}\Spin(8p+8)$. Conversely, given a spin structure $\PSpin(M\times \RR^m)$ on $M\times \RR^m$, $i^*(\PSpin(M\times \RR^m))$ is a $\Spin(8p+8)$ structure. Taking the preimage of $\PSO(M)\subset \mathrm{P}_{\SO(8p+8)}(M)$,
we get a spin structure on $M$. 

Moreover, there is an isomorphism between the bundles $\PSpin(M)\times_{\Spin(n)}W^+$ and $\PSpin(M)\times_{\Spin(n)}\Spin(8p+8)\times_{\Spin(8p+8)} W^+$, given by $ [\tilde{F},v]\mapsto [[\tilde{F},1],v]$. Thus, taking into account Proposition \ref{restriction}, we get $\S^+(M\times \RR^m)=\pr_1^*\S(M)$.

The relation between Clifford products is follows from the equality $\rho_n(v)=\rho_{8p+8}(ve_{n+1})$, for $v \in \RR^n$; this is obtained using the definition of $\rho_n$ in Proposition \ref{restriction} as follows:
\[
\rho_n(v)=\rho_{n+1}(ve_{n+1})=\rho_{n+2}(ve_{n+2}e_{n+1}e_{n+2})=\rho_{n+2}(ve_{n+1})=\dots =\rho_{8p+8}(ve_{n+1})\,.
\]

\end{proof}

The scalar product $\la\cdot, \cdot\ra$ on $W$ defines a scalar product on the spinor bundle that we also denote by $\la\cdot, \cdot\ra$; the Clifford product with a vector field is a skew-symmetric endomorphism. The Levi-Civita connection $\overline{\nabla}$ of $g$ induces a connection $\nabla$ on the spinor bundle which is $\la\cdot, \cdot\ra$-metric and acts as a derivation with respect to the Clifford product with a vector field. Moreover, the complex and quaternionic structures on $W$ determine complex and quaternionic structures on the spinor bundle, which are isometries of $\la\cdot, \cdot\ra$ and parallel with respect to $\nabla$.

\begin{definition} \label{Def-Dirac}
The {\em Dirac operator} is the differential operator $\sD\colon \Gamma(\S(M)) \to \Gamma (\S(M))$ given locally by the expression
\[
\sD\p = \sum_{i=1}^n{X_i \nabla_{X_i}\p}\,, 
\]
where $(X_1,\dots, X_n)$ is a local orthonormal frame of $M$.
\end{definition}

\begin{definition} \label{Def-harmonic}
A spinor $\eta\in\Gamma(\S(M))$ is called {\em harmonic} if $\sD\eta=0$.
\end{definition}

There is a relation between positive harmonic spinors in different dimensions; we follow the notation of Lemma \ref{Spinor-Bundle-Relation}:

\begin{lemma}\label{Harmonic-Spinors-Relation}
For $m\in \{1,2,3,4\}$, let $(M,g)$ be an $(8p-m)$-dimensional spin Riemannian manifold.
Let $\p$ be a unit-length harmonic spinor of $M$. Then, $\eta=\pr_1^*\p$ is a unit-length harmonic spinor on $M\times \RR^m$.
\end{lemma}
\begin{proof}
Take $(X_1,\dots,X_n)$ a local orthonormal frame of $TM$ and $(e_{n+1},\dots, e_{8p+8})$ an orthonormal basis of $\RR^m$; observe that $\overline{\nabla}{}^{M\times \RR^m}_{X_i}X_j=\overline{\nabla}{}^{M}_{X_i}X_j$, $\overline{\nabla}{}^{M\times \RR^m}_{e_i}{X_j}=\overline{\nabla}{}^{M\times \RR^m}_{X_j}{e_i}=0$ and $\overline{\nabla}{}^{\RR^m}_{e_i}e_j=0$. Therefore, $\nabla^{M\times \RR^m}_{X_i}\eta
=\pr_1^*(\nabla^M_{X_i}\p)$ and $\nabla^{M\times \RR^m}_{e_i}\eta=0$. From the relation between $\S(M)$ and $\S^+(M\times \RR^m)$ proved in Lemma \ref{Spinor-Bundle-Relation} we deduce:
\[
e_{n+1}\sD\eta= \sum_{i=1}^{n} e_{n+1}X_i \nabla^{M\times \RR^m}_{X_i}\eta= - \pr_1^*\sD\p \,.
\]
The spinor $\eta$ is harmonic because the multiplication by $e_{n+1}$ is an isometry.
\end{proof}


\section{Spinors and geometric structures}\label{Spinors}

The purpose of this paper is to study geometric structures defined by unit-length harmonic spinors on Riemannian manifolds. 
This is interesting because a unit-length harmonic spinor defines different geometric structures according to the dimensions. We shall focus on dimensions 4, 5, 6, 7 and 8. In these dimensions, the relation between unit-length spinors and geometric structures on manifolds is summarized in the following result:
\begin{proposition}
Let $\rho_n \colon \Cl_n \to \End_\bk(W)$ an irreducible representation and let $\eta \in W$ be a unit-length spinor.
\begin{enumerate}
\item If $n=8$ and $\eta \in W^\pm$ then $\Stab_{\Spin(8)}(\eta)= \Spin(7)$.
\item If $n=7$ then $\Stab_{\Spin(7)}(\eta)= \rG_2$.
\item If $n=6$ then $\Stab_{\Spin(6)}(\eta)= \SU(3)$.
\item If $n=5$ then $\Stab_{\Spin(5)}(\eta)= \SU(2)$.
\item If $n=4$ then $\Stab_{\Spin(4)}(\eta)= \SU(2)$.
\end{enumerate}
\end{proposition}

This proposition means that a unit-length spinor in dimension 8 determines a $\Spin(7)$ structure on the underlying manifold, and analogously for the other dimensions.

Motivated by Definition \ref{Def-harmonic}, we give the following definition:
\begin{definition} \label{harmonic-structure}
Let $(M,g)$ be a Riemannian spin manifold of dimension $n \in\{4,\ldots,8\}$, and let $\eta \in \G(\S(M))$ be a unit-length section. We say that $\eta$ determines a {\em spin-harmonic structure} on $M$ if $\sD\eta=0$.
Moreover, if $n\equiv 0 \pmod{4}$, we say that the spin-harmonic structure is {\em positive} or {\em negative} if $\eta \in \G(\S^\pm(M))$.
\end{definition}

\begin{remark}
For dimensions $n>8$, the action of $\Spin(n)$ on the sphere of unit-length spinors is not transitive. Therefore the stabilizers of the spinors may be different groups,
so it makes no sense to define a geometric structure via a unit-length spinor unless we require the constancy of the stabilizer (this happens for instance when one has
a parallel spinor).
\end{remark}

From now on, we denote a generic spinor by $\phi$ and a fixed unit-length spinor by $\eta$. 

More specifically, our motivation is constructing $8$-dimensional nilmanifolds with invariant balanced $\Spin(7)$ structures. As we shall see later, these structures are characterized by the presence of a positive spin-harmonic structure. Lemma \ref{Harmonic-Spinors-Relation} guarantees that if $n\in \{4,5,6,7\}$, $M$ is an $n$-dimensional spin manifold with a spin-harmonic structure and $T^{8-n}$ is an $(8-n)$-dimensional flat torus, then $M\times T^{8-n}$ has a $\Spin(7)$ balanced structure. In section \ref{sec:harm-nilm} we will construct such spin-harmonic structures on low dimensional nilmanifolds.

Spin-harmonic structures have already appeared, under disguise, in the papers \cite{ACFH15} and \cite{MM18}; we proceed to review the relevant results and to relate spin-harmonic structures with the different kinds of $\Spin(7)$, $\rG_2$ and $\SU(3)$ structures. There is no spinorial description of $\SU(2)$ structures in dimension 5; we will carry out this classification in Section \ref{sec:su(2)}. We will not study the condition in dimension $4$; in fact, as we shall see in Theorem 
\ref{4-harmonic}, there are no invariant harmonic spinors on $4$-dimensional nilmanifolds.

\subsection{Positive spin-harmonic \texorpdfstring{$\Spin(7)$}{Lg} structures in dimension 8}\label{subsec:spin(7)}

Let $(M,g)$ be an $8$-dimensional Riemannian manifold; a $\Spin(7)$ structure is characterized by the presence of a triple cross product on each tangent space; in turn, this is determined by a $4$-form $\Omega$ (see \cite[Definition 6.13]{SW}).

As usual, a way to measure the lack of integrability of a geometric structure is provided by its intrinsic torsion (see \cite{Salamon89}). In this case, the intrinsic torsion of a $\Spin(7)$ structure is a section of the bundle $T^*M\otimes \frspin(7)^\perp$, which is isomorphic to $\L^3T^*M$ via the alternating map. The Hodge star defines an isomorphism $\star\colon \L^3 T^*M\to \L^5 T^*M$. Therefore, the different classes of $\Spin(7)$ structures are determined by the exterior derivative of $\Omega$.

For a fixed $\Spin(7)$ form $\Omega$ on $\RR^8$, the decomposition of the space of $3$-forms of $\RR^8$ into irreducible $\Spin(7)$ invariant subspaces is given by (see \cite[Theorem 9.8]{SW}):
\begin{align*}
\L^3 (\RR^8)^* = & \L^3_8(\RR^8)^* \oplus \L^3_{48}(\RR^8)^*.
\end{align*}
where $\L^3_8(\RR^8)^*=i(\RR^8) \Omega$ and $\L^3_{48}(\RR^8)^*=\{\t\in\L^3(\RR^8)^* \mid \tau \wedge\Omega=0\}$. We have denoted by $\L^k_l(\RR^8)^*$ an $l$-dimensional invariant subspace of $\L^k(\RR^8)^*$; moreover, the induced bundle on $M$ will be denoted by $\L^k_l T^*M$. According to this discussion, there exist $\t_1 \in \L^1T^*M$ and $\t_3 \in \L^3_{48}T^*M$ such that:
\[
d\Omega=\t_1\wedge \O + \star\, \t_3\,. 
\]

In \cite{MF86}, Fern\'andez distinguished $\Spin(7)$ structures in the following pure classes:

\begin{definition}\label{spin7-clasif}
A $\Spin(7)$-structure given by $\Omega$ is said to be:
\begin{enumerate}
\item {\em parallel}, if $d\Omega=0$;
\item {\em locally conformally parallel}, if $\tau_3=0$;
\item {\em balanced}, if $\tau_1=0$.
\end{enumerate} 
\end{definition}

A Riemannian manifold $(M,g)$ admitting a $\Spin(7)$ structure is spin and the positive part of its spinor bundle has a unit-length section. Conversely, a spin $8$-dimensional manifold whose spinor bundle admits a positive unit-length section $\eta$ can be endowed with a $\Spin(7)$ structure by the formula
\[
\Omega(W,X,Y,Z)=\frac{1}{2} \la(-WXYZ + WZYX)\eta,\eta\ra\,.
\]

As for spin-harmonic structures, the following result was proved by the second author in \cite{MM18}:
\begin{theorem}\label{Spin7-Dirac}
The spinor $\eta$ determines a positive spin-harmonic structure if and only if the induced $\Spin(7)$ structure is balanced.
\end{theorem}

\begin{remark}
 Spin-harmonic structures are thus especially relevant in dimension 8, since they represent a pure class of $\Spin(7)$ structures.
\end{remark}

\subsection{Spin-harmonic \texorpdfstring{$\rG_2$}{Lg} structures in dimension 7}\label{subsec:G2}

A $\rG_2$ structure on a Riemannian $7$-dimensional manifold $(M,g)$  is characterized by the presence of a cross product on $(TM,g)$, which is determined by a $3$-form $\Psi$ (see \cite[Lemma 2.6]{SW})

The torsion of a $\rG_2$ structure is a section of the bundle $T^*M\otimes \frg_2^\perp$. The splitting of $\RR^7\otimes \frg_2^\perp$ into four $\rG_2$ invariant irreducible subspaces determines four subbundles, $\chi_1,\chi_2,\chi_3,\chi_4$ which, in turn, determine pure types of $\rG_2$ structures.

Such classes are completely determined by differential equations for $\Psi$ and $*\Psi$. In order to state the precise result, we recall the decomposition of $\L^2 (\RR^7)^*$ and $\L^3(\RR^7)^*$ into $\rG_2$ irreducible parts for a fixed $\rG_2$ form $\Psi$ of $\RR^7$ (see \cite[Theorem 8.5]{SW}):
\begin{align*}
\L^2 (\RR^7)^* =& \L^2_7 (\RR^7)^* \oplus \L^2_{14}(\RR^7)^*,\\
\L^3 (\RR^7)^* = & \L_1^3(\RR^7)^*  \oplus \L^3_7 (\RR^7)^* \oplus \L^3_{27}(\RR^7)^*\,,
\end{align*}
where $\L^2_7 (\RR^7)^*=i(\RR^7)\Psi$, $\L^2_{14}(\RR^7)^*=\frg_2$, $\L_1^3(\RR^7)^*=\la \Psi \ra$, $\L^3_7 (\RR^7)^*= i(\RR^7)(\star\Psi)$ and $\L^3_{27}(\RR^7)^*=\{\omega \mid \Psi\wedge\omega=0,\star\Psi\wedge\omega=0\}$. Then we have (see \cite[Proposition 1]{Bryant}):

\begin{proposition}\label{G2-Torsions}
There exist $\t^1 \in C^\infty(M)$, $\t^4 \in \L^1 T^*M$, $\t^2 \in \L^2_{14}T^*M$ and $\t^3 \in \L^3_{27}T^*M$ such that: 
\begin{align*}
d\Psi=& \t^1(\star\Psi) + 3 \t^4\wedge \Psi +\star\t^3,\\
d(\star\Psi)=& 4\t^4\wedge(\star\Psi) + \t^2 \wedge \Psi.
\end{align*}
Moreover, the torsion is a section of $\chi_j$ if and only if $\t^k=0$ for $k\neq j$.
\end{proposition}

A Riemannian manifold $(M,g)$ admitting a $\rG_2$ structure is spin and its spinor bundle has a unit-length section. Conversely, the spinor bundle $\S(M)$ of a spin $7$-manifold $M$ has a unit-length section $\eta$ and the $3$-form of the $\rG_2$ structure is given by \cite{ACFH15}:
\[
\Psi(X,Y,Z)=\la XYZ\eta,\eta\ra\,.
\]

The relationship between $\rG_2$-structures and harmonic spinors is characterized by the following result:
\begin{theorem} \cite[Theorem 4.8]{ACFH15} \label{G2-Dirac}
The spinor $\eta$ determines a spin-harmonic structure if and only if the induced $\rG_2$ structure is of type $\chi_{2}\oplus\chi_3$.
\end{theorem}

\subsection{Spin-harmonic \texorpdfstring{$\SU(3)$}{Lg} structures in dimension 6}\label{subsec:su(3)}

Let $(M,g)$ be a $6$-dimensional Riemannian manifold. An $\SU(3)$ structure on $M$ consists in a compatible almost complex structure $J$ and a complex volume form $\Theta$ (see \cite{Hitchin,Salamon89}). We denote by $\Theta_+$ and $\Theta_-$ the real and imaginary part of $\Theta$ and we define the fundamental 2-form $\omega$ by $\omega(X,Y)=g(JX,Y)$ for $X,Y\in\fX(M)$. 

The space $\RR^6\otimes \frsu(3)^\perp$ decomposes into seven $\SU(3)$-invariant irreducible subspaces; accordingly the intrinsic torsion of an $\SU(3)$ structure, which is a section of $T^*M\otimes \frsu(3)^{\perp}$, decomposes into the subbundles $\chi_1$, $\chi_{\bar{1}}$, $\chi_2$, $\chi_{\bar{2}}$, $\chi_3$, $\chi_4$, $\chi_5$ (see \cite{CS02}).

These are related to differential equations for $\omega$, $\Theta_{+}$ and $\Theta_{-}$. Before formulating the result, we recall the decomposition of $\L^2(\RR^6)^*$ and $\L^3(\RR^6)^*$ into $\SU(3)$ irreducible representations. For this, we consider the $\UU(3)$ decomposition $\L^n (\CC^6)^* = \oplus_{p+q=n} \L^{p,q}(\CC^6)^*$ and we denote the real part of a complex vector space $V$ by $\llbracket V\rrbracket$. For a fixed $\SU(3)$ structure $(\o,\Theta_+,\Theta_-)$ on $\RR^6$, the splitting is:
\begin{align*}
\L^2 (\RR^6)^* =& \la \o \ra \oplus \llbracket\L^{1,1}_0 (\CC^6)^* \rrbracket \oplus i(\RR^6) \Theta_{+}, \\
\L^3 (\RR^6)^* = & \la \Theta_{+} \ra \oplus \la \Theta_{-} \ra \oplus \llbracket\L^{2,1}_0 (\CC^6)^*\rrbracket\oplus \RR^6 \wedge \o_0.
\end{align*}
where  $\L^{1,1}_0 (\CC^6)^*$ and $\L^{2,1}_0 (\CC^6)^*$ are the spaces of primitive forms, that is, forms of $\L^{1,1}(\CC^6)^*$ and  $\L^{2,1}(\CC^6)^*$ which are orthogonal to $\o$ and $\o \wedge (\CC^6)^*$, respectively. 
The associated bundles of $M$ will be denoted respectively by $\llbracket \L^{1,1}_0 (T^*M\otimes \CC)\rrbracket$ and $\llbracket\L^{2,1}_0 (T^*M\otimes\CC)\rrbracket$.

\begin{proposition}\cite[Section 2.5]{BV07}\label{SU(3)-Torsions}
There exist $\t^{1},\tau^{\bar{1}} \in C^\infty(M)$, $\t^4,\t^5\in \L^1 T^*M$, $\t^{2},\t^{\bar{2}} \in\llbracket\L^{1,1}_0 (T^*M\otimes \CC)\rrbracket$ and $\t^3 \in\llbracket\L^{2,1}_0 (T^*M\otimes\CC) \rrbracket$ such that:
\begin{align*}
d\omega   & =  -\frac{3}{2}\t^{\bar{1}}\Theta_+ + \frac{3}{2} \t^{1} \Theta_{-} + \t^{3} + \t^4 \wedge \omega,\\
d\Theta_+ & = \t^{1} \omega^2 - \t^{2} \wedge \omega+ \t^5 \wedge \Theta_+,\\
d\Theta_- & = \t^{\bar{1}} \omega^2 - \t^{\bar{2}} \wedge \omega+ J\t^5 \wedge \Theta_+\,.
\end{align*}
Moreover, the intrinsic torsion is a section of $\chi_j$ if and only if $\t^k=0$ for $k\neq j$.
\end{proposition}

A Riemannian manifold $(M,g)$ with an $\SU(3)$ structure is spin and its spinor bundle has a unit-length section. Conversely, a spin $6$-dimensional manifold has a unit-length spinor; the following proposition explains how the spinor induces the $\SU(3)$ structure.

\begin{proposition}\cite[Section 2]{ACFH15}\label{su3-spinors}
The spinor bundle of $M$ splits as
\[
\S(M)=\la \eta \ra \oplus \la \jota\eta \ra \oplus  TM \eta\,.
\]
The fundamental form $\omega$ and the real part of the complex 3-form $\Theta_+$ of the $\SU(3)$ structure determined by $\eta$ are given by 
\[
 \omega(X,Y)=\la \jota X\eta,Y\eta\ra \quad \textrm{and} \quad \Theta_+=-\la XYZ\eta,\eta\ra\,.
\]
\end{proposition}

Proposition \ref{su3-spinors} guarantees the existence and uniqueness of $S\in \End(TM)$ and $\g \in T^*M$ such that:
\[ 
\nabla_X\eta= S(X)\eta + \g(X)\jota\eta\,.
\]

The relation between harmonic spinors and $\SU(3)$ structures is  given by the following result:
\begin{theorem} \cite[Theorem 3.7]{ACFH15} \label{SU(3)-Dirac}
The spinor $\eta$ determines a spin-harmonic structure if and only if its induced $\SU(3)$ structure is in the class $\chi_{2\bar{2}345}$ and verifies $\delta \omega = -2\g$.
\end{theorem}

We finally relate Theorem \ref{SU(3)-Dirac} and Proposition \ref{SU(3)-Torsions}.

\begin{corollary}
The $\SU(3)$ structure is spin-harmonic if and only if it lies in $\chi_{2\bar{2}345}$ and satisfies $\t^4= \t^5$.
\end{corollary}
\begin{proof}
First, $\delta\o =  -\star(\t^4\wedge \o^2)= J\t^4$. 
To find an expression for $\gamma$ in terms of the torsion forms we first observe that, according to 
\cite[Theorem 3.13]{ACFH15}, it only depends on the projection of the intrinsic torsion $\Gamma$ to $\chi_5$. Therefore, we assume that $\gamma \in \chi_5$ for this computation; observe that in this case $d\Theta_+=\t^5 \wedge \Theta_+$, due to Proposition \ref{SU(3)-Torsions}. 

If the torsion lies in $\chi_5$ then, $\nabla_X \eta = \g(X)\jota\eta$ and therefore, for orthonormal vectors: 
 $\nabla_W\Theta_+ (X,Y,Z)=-2\g(W)\la XYZ\eta,\jota\eta\ra$
$=2\g(W)\la J(X)YZ\eta,\eta\ra=-2\g(W)\Theta_-(X,Y,Z),
$

where we used that $X \jota \eta=-\jota X \eta=-J(X)\eta$ and that $\Theta_-(X,Y,Z)=\Theta_+(J(X),Y,Z)$.
Therefore, 
\begin{align*}
d\Theta_{+}& (W,X,Y,Z) = \\ &= \nabla_W\Theta_+(X,Y,Z)-\nabla_X\Theta_+ (W,Y,Z )+\nabla_Y\Theta_+(X,W,Z)
-\nabla_Z\Theta_+(X,Y,W) \\ 
&=  -2 \g\wedge\Theta_-(W,X,Y,Z).
\end{align*}
In addition, one can observe that $\a \wedge \Theta_-=-J\a\wedge \Theta_+$ for $\a \in \xi^*$; this implies that, $\t^5=2J\gamma$. Therefore, the equality $\delta \omega= -2\gamma$ is equivalent to $\t^4=\t^5$.
\end{proof}


\section{Spin-harmonic \texorpdfstring{$\SU(2)$}{Lg} structures on 5-dimensional manifolds}\label{sec:su(2)}

\subsection{\texorpdfstring{$\SU(2)$}{Lg} structures}\label{subsec:su(2)-structures}

An $\SU(2)$ structure on a Riemannian manifold $(M,g)$ is determined by an orthogonal splitting $TM= \xi \oplus \la \a^\sharp \ra$, where $\a$ is a unit-length $1$-form and the distribution $\xi=\ker\a$ is endowed with three almost complex structures $J_k \colon \xi \to \xi$, $k=1,2,3$ which are isometries with respect to the induced metric, and satisfy $J_1\circ J_2=J_3$ and $J_k\circ J_l=-J_l\circ J_k$ for $k\neq l$. The vector field $\a^\sharp$ is denoted by $R$. The three fundamental $2$-forms are given by $\o_k(X,Y)=g(J_kX,Y)$, $k=1,2,3$, $X,Y\in\fX(M)$.

In fact, $\SU(2)$ structures are characterized by the forms $(\a,\o_1,\o_2,\o_3)$, as the following result states:

\begin{proposition}\cite[Proposition 1]{CS07}\label{SU2-forms} 
$\SU(2)$ structures on a $5$-manifold are in one-to-one correspondence with $(\a,\o_1,\o_2,\o_3) \in \L^1 T^*M \times (\L^2 T^*M)^3$, such that:
\begin{enumerate}
\item $\o_i \wedge \o_j=0$ for $i\neq j$, $\o_1^2 = \o_2^2= \o_3^2$ and $\a \wedge \o_1^2 \neq 0$,
\item If $i(X)\o_1= i(Y)\o_2$, then $\o_3(X,Y)\geq 0$.
\end{enumerate}
\end{proposition}

\begin{proposition} \label{local-form} \cite[Corollary 3]{CS07}
Let  $(\a,\o_1,\o_2,\o_3)$ be an $SU(2)$ structure on a 5-manifold. There is a local frame of the cotangent bundle, $(e^1,\dots, e^5)$, such that $\a=e^5$, $\o_1=e^{12}+e^{34}$, $\o_2= e^{13}-e^{24}$, $\o_3=e^{14} + e^{23}$. 
\end{proposition}

An almost complex structure $J_k\colon\xi\to\xi$ defines an almost complex structure on $\xi^*$ by $(J_k\beta)(X)=
\beta(J_kX)$ for $\b \in \xi^*$ and $X\in\xi$; one has $(J_k\circ J_l)\beta=(J_l\circ J_k)\b$, but $(J_1\circ J_2)\b= -J_3\b$. The next lemma will be used in the next section:

\begin{lemma}\label{hodge-star}
For $\b \in \xi^*$, $\star_\xi(\b \wedge \o_k)= -J_k\b$.
\end{lemma}
\begin{proof}
We compute the equality for $\b=e^1$. Using that $J_ke^1=-(J_ke_1)^*$ and that $\o_k= -( \rI + \star_\xi)(e^1\wedge J_ke^1)$, we get:
$\star_\xi (e^1\wedge \o_k)= -\star_\xi(e^1 \wedge\star_\xi(e^1 \wedge J_ke^1))= -(i(e_1)(e^1\wedge J_ke^1))=-J_ke^1$.
\end{proof}

As usual, $\SU(2)$ structures are classified by the intrinsic torsion, which is a section of $T^*M\otimes \frsu(2)^\perp$. In the following, we denote the intrinsic torsion by an $\SU(2)$ equivariant map,
\[
\Xi \colon \PSO(M) \to T^*M \otimes \frsu(2)^\perp,
\]
where $\PSO(M)$ is the frame bundle of $M$. Proposition \ref{differentials-SU2} below shows that $\Xi$ is determined by $(d\a,d\o_1,d\o_2,d\o_3)$. In order to state it, we recall the irreducible decomposition of some $\SU(2)$ modules (see \cite{BV09}).

\begin{proposition}\label{irreducible-decomposition} 
Let $\RR^5$ be endowed with the $\SU(2)$ structure $(\a,\o_1,\o_2,\o_3)$. Then
\begin{enumerate}
\item $\L^1 (\RR^5)^*= \la \a \ra \oplus \xi^*$,
\item $\L^2 (\RR^5)^*= \a \wedge \xi^* \oplus (\oplus_{k=1}^3\la \o_k \ra) \oplus \frsu(2)$,
\item $\L^3 (\RR^5)^*= \L^3 \xi^*  \oplus (\oplus_{k=1}^3\la \a \wedge \o_k \ra) \oplus \a \wedge \frsu(2)$,
\item $\End(\xi)= \la \rI \ra \oplus (\oplus_{k=1}^3 \s_k(\xi)) \oplus (\oplus_{k=1}^3\la J_k \ra) \oplus \frsu(2)$, where 
\[\s_k(\xi)=\left\{S\in \Sym_0(\xi)\mid SJ_l=(-1)^{\delta_k^l+1}J_l S, \ l=1,2,3\right\}, \ k=1,2,3\,.
\]
\end{enumerate}
Moreover, the map $E_k \colon \s_k(\xi) \to \frsu(2)$, $E_k(S)=i(S)\o_k$ is an isomorphism.
\end{proposition}


\begin{proposition} \label{differentials-SU2}  \cite[Proposition 9]{CS07}
As an $\SU(2)$-module, $\RR^5\otimes\frsu(2)^\perp$ decomposes as:
\[
\RR^5\otimes \frsu(2)^\perp=  7\RR \oplus 4(\RR^4)^* \oplus 4\frsu(2),
\]
where $7\RR$ means 7 copies of the trivial representation $\RR$, and so on.
Let $\tau_0^l,\tau_0^{kl} \in C^{\infty}(M)$, $k,l=1,2,3$, $\tau_1^k\in \xi^*$ and $\tau_2^k \in \frsu(2)$, $k=1,2,3,4$, be such that
\begin{align*}
d\a = &  \sum_{l=1}^3\tau_0^l\o_l + \a \wedge \tau_1^4 + \tau_2^4\, , \\
d\o_k =& \sum_{l=1}^3\tau_0^{kl}\a \wedge \o_l + \tau_1^k \wedge \o_k + \a \wedge \tau_2^k\,, 
\end{align*}
Then $\tau_0^{kk}=\tau_0^{ll}$ and $\tau_0^{kl}=-\tau_0^{l k}$ for $l \neq k$. Moreover,
\[
\Xi(u)=((\tau_0^{11},\tau_0^{jk},\tau_0^{l}),(u^*\tau_0^j,u^*\tau_1^4),(u^*\tau_0^j,u^*\tau_2^4))\,.
\]
\end{proposition}

\subsection{Spinorial point of view}\label{subsec:spinors}

Let $\rho_5 \colon \Cl_5 \to \End_\CC(W)$ be an irreducible representation with complex structure $\jota_1=\rho_5(\nu_5)$. Take also a quaternionic stucture $\jota_2$ that anticommutes with the Clifford product (see Propositon \ref{real-and-quaternionic-structures}), and define $\jota_3=\jota_1\circ \jota_2$. For our purposes we shall define $\e_1=1$ and $\e_2=\e_3=-1$; we have that:   
$
\jota_k X \p= \e_k X \jota_k \p,
$
for every spinor $\p$.

Let $(M,g)$ be a spin Riemannian manifold and let $\Ad\colon \mathrm{P}_{\Spin(5)}M \to \mathrm{P}_{\SO(5)}M$ be a spin structure. The spinor bundle $\S(M)=\mathrm{P}_{\Spin(5)}(M)\times_{\rho_5}W$ has a unit-length section $\eta$. Define $\Stab(\eta)$ as the subbundle whose fiber at $p\in M$ is the stabilizer of the spinor $\eta(p)$ under the action of $\Spin(5)$. It is an $\SU(2)$ reduction of $\mathrm{P}_{\Spin(5)}(M)$, and the projection $\Ad(\Stab(\eta))$ is an $\SU(2)$ structure because the kernel of $\Ad$ is $\pm 1$ and $-1\notin \Stab(\eta_p)$.

We first explain the decomposition of the spinor bundle of $M$ and write the forms that determine the structure by means of spinors. For that purpose consider the map $\rho_\eta \colon \Spin(5) \to W$, $\rho_\eta(g)=g\eta$, whose differential is $d\rho_\eta\colon \L^2 \RR^5 \to W $, $d\rho_\eta(\gamma)=\gamma \eta $. 

\begin{lemma} \label{decomposition-spinorial-bundle}
The restriction $d\rho_\eta \colon \frsu(2)^\perp \to \la \eta \ra^\perp$ is an isomorphism, 
hence there is a decomposition of $\la \eta \ra ^\perp$ with respect to the $\SU(2)$ structure determined by $\eta$, $(\a, \omega_1,\omega_2,\omega_3)$:
 \[
 \S(M)= \la \eta \ra \oplus( \oplus_{k=1}^3 \la \o_k \eta \ra )\oplus \xi^* \eta. 
 \]
\end{lemma}

\begin{proof}
The kernel of $d\rho_\eta$ is $\frsu(2)$ because $\Stab(\eta)=\SU(2)$ and $\im d\rho_\eta \subset \la \eta \ra^\perp$.
By Proposition \ref{irreducible-decomposition}(2), we have 
$\S(M)= \la \eta \ra \oplus( \oplus_{k=1}^3 \la \o_k \eta \ra )\oplus (\a \wedge \xi^*) \eta$. Now
$(\a\wedge \xi^*)\eta=\xi^*\eta$ because these are irreducible representations of the same dimension.
\end{proof}

We can write the forms that determine the $\SU(2)$ structure in terms of spinors.

\begin{lemma} \label{decomposition-2}
The spinors $\eta$, $\jota_1\eta$, $\jota_2\eta$, $\jota_3 \eta$ are orthogonal and the spaces $\HH_\eta= \la \eta, \jota_1\eta, \jota_2\eta, \jota_3 \eta  \ra$ and $\HH_\eta^\perp$ are $\jota_k$-invariant, $k=1,2,3$.

Moreover, there exists a subspace $\xi \subset \RR^5$ such that $\xi\eta=\HH_\eta^\perp$; $\xi$ inherits a quaternionic structure determined by $J_k(X)\eta= \jota_k(X\eta)$.
\end{lemma}
\begin{proof}
 The orthogonality of the mentioned spinors follows from the fact that the endomorphisms $\jota_k$ are isometries. From this property it also follows that the subspace $\HH_\eta ^\perp$ is $\jota_k$-invariant.
 
In addition, $\HH_\eta ^\perp$ is $\SU(2)$-irreducible as a consequence of Lemma \ref{decomposition-spinorial-bundle}, and the map $X \mapsto  X\eta$ is injective and $\SU(2)$-equivariant. Since 
$\RR^5=\RR \oplus \CC^2$ as $\SU(2)$ modules, necessarily $\HH_\eta^\perp = \xi \eta$ for some $\xi \subset \RR^5$.
Finally, the endomorphisms $J_k$ define a quaternionic structure on $\xi$, since $\jota_2$ is a quaternionic structure on $\HH_\eta^\perp$.
\end{proof}

\begin{definition} \label{SU2-forms-spinors}
Let $(M,g)$ be a Riemannian manifold with a spin structure and let $\eta\in\S(M)$ be a unit spinor. The $\SU(2)$ structure $(\alpha,\o_1,\o_2,\o_3)$ defined by $\eta$ is given by:
\begin{enumerate}
 \item $\o_k(X,Y)= g(J_k X_\xi,Y_\xi)$, where $Z_\xi$ is the orthogonal projection of a vector field $Z$ to $\xi$.
 \item $\RR^5\cong \xi \oplus \la R \ra$ as oriented vector spaces, where $\xi$ is oriented by $\o_1^2|_\xi$, and $R=\alpha^\sharp$.
\end{enumerate}
\end{definition}

\begin{lemma} \label{su2-equalities}
The following equalities hold:
\begin{enumerate}
\item $ \o_k \eta= -2\e_k \jota_k\eta$, with $\e_1=1$ and $\e_2=\e_3=-1$,
\item $\a \eta = - \jota_1\eta$,
\item $\a \jota_2 \eta =-\jota_3 \eta$ and $\a \jota_3\eta=\jota_2\eta$.
\item $\nu \eta= -\jota_1\eta$, where $\nu$ is the positively-oriented unit-length volume form.
\end{enumerate}
\end{lemma}
\begin{proof}
Take an orthonormal oriented frame $(e_1,e_2,e_3,e_4,e_5)$ such that $\o_1=e^{12} + e^{34}$, $\o_2= e^{13} - e^{24}$, $\o_3=e^{14} + e^{23}$ and $\a=e^5$.
Since $J_1(e_1)=e_2$ and $J_1(e_3)=e_4$,
\[
\o_1 \eta= (e_1e_2 + e_3e_4)\eta = e_1J_1(e_1)\eta + e_3J_1(e_3)\eta= \jota_1(e_1^2 + e_3^2)\eta = -2\jota_1\eta\,.
\]
For $k\in \{2,3\}$ the computation is similar, but one has to take into account that $\jota_2$ and $\jota_3$ anticommute with the Clifford product with a vector.

Finally, $e_{12}\eta=-\jota_1\eta=e_{34}\eta$ implies $\nu\eta = -e_5\eta$. The second and third equalities are a consequence of the latter one, together with the fact that $\jota_1\jota_2=\jota_3$. For instance, $\a \jota_2 \eta =-\jota_2\a \eta = \jota_2\jota_1\eta=-\jota_3\eta$.

For the last equality, observe that in terms of the previous frame we have:
$\nu=e^{12345}=e^1\wedge (J_1(e_1))^*\wedge e^3 \wedge (J_1(e_3))^* \wedge e^5$. Taking into account the previous equalities and that $(e^k\wedge(J_1(e^k))^*)\eta = -\jota_1\eta$ for $k\in \{1,3\}$ as before, we obtain:
$$
\nu \eta=-\jota_1 e^1J_1(e^1)e^3J_1(e^3)\eta=-\jota_1\eta.
$$
\end{proof}

\begin{remark} \label{orthogonal-clifford-product}
The subspaces $\L^2 \xi^* \eta$ and $\xi^* \eta$ are orthogonal.
\end{remark}

\begin{lemma} 
For $\e_1=1$ and $\e_2=\e_3=-1$, $\o_k(X,Y)=\e_k\la X\jota_k\eta,Y\eta\ra$.
Moreover, $\a(X)=-\la X\eta,\jota_1\eta\ra$.
\end{lemma}
\begin{proof}
The tensor $(X,Y)\mapsto\la X\jota_k\eta,Y\eta\ra$ is skew-symmetric because $\jota_k$ is an isometry, $\jota_k^2=-\rI$ and $\la \jota_k \eta,\eta\ra=0$. If $X,Y\in \xi$,
\[
\o_k(X,Y)=  g(J_kX,Y)=\la J_kX\eta,Y\eta\ra=\e_k\la X\jota_k\eta,Y\eta\ra.
\] 
Moreover, $\o_k(R,Y)= 0=\e_k\la R \jota_k\eta, Y\eta\ra$, because $R \jota_k\eta \in \HH_{\eta}$ and $Y\eta \in \HH_{\eta}^\perp$. Finally, $\a(X)=\la X\eta, R\eta\ra=-\la X\eta,\jota_1\eta\ra$.
\end{proof}

Our next purpose is to compute the Dirac operator of $\eta$ in order to relate it with the torsion of the $\SU(2)$ structure. We first introduce some notation.

\begin{definition} \label{covariant-derivative} 
Lemmas \ref{decomposition-spinorial-bundle} and \ref{su2-equalities} guarantee the existence and uniqueness of $S\in \End(\xi)$, $V_\xi \in \xi$, $\Theta_l \in \xi^*$ and $\p_l \in C^\infty(M)$, $l=1,2,3$, such that:
\begin{equation}\label{eqn:1}
\nabla_X \eta = S(X_\xi)\eta + \a(X) V_\xi \eta + \sum_{l=1}^3{(\Theta_l(X_\xi) + \a(X)\p_l)\jota_l\eta}\,,
\end{equation}
where $X=X_{\xi}+ \a(X)R$.
\end{definition}

\begin{definition}\label{endomorphism-S} According to Proposition \ref{irreducible-decomposition}, there is a decomposition of $S\in \End(\xi)$:
\[
S(X)= \mu\, \rI + \sum_{l=1}^3{S_l} + \sum_{l=1}^3{\l_l}J_l + S_0\, ,
\]
where $S_k \in \s_k(\xi)$ and $S_0 \in \frsu(2)$.
\end{definition}

We now compute the Dirac operator of $\eta$ in terms of the tensors we introduced; we use the notation of Definition \ref{covariant-derivative}.
\begin{proposition} \label{Dirac-su2-spinors}
Let $\eta \in \S(M)$ be a unit-length spinor. The Dirac operator is
\begin{align*}
\sD\eta = & (-4\mu + \p_1)\eta - 4\l_1 \jota_1\eta + (4\l_2 + \p_3)\jota_2\eta + (4\l_3 - \p_2)\jota_3\eta \\
 & + (J_1(V_{\xi} + \Theta_1^\sharp)- J_2(\Theta_2^\sharp)- J_3(\Theta_3^\sharp)) \eta.
\end{align*} 
\end{proposition}

\begin{proof}
Let $(e_1,\dots,e_4,R)$ be an oriented orthonormal local frame. From (\ref{eqn:1}), we have
\[
 \sD\eta= m(S) + R V_{\xi}\eta + \sum_{k=1}^3 \left(\left(\sum_{i=1}^4 \Theta_k(e_i) e_i \right) + \phi_k R\right)\jota_k\eta \,,
\]
where $m \colon \End(\xi) \to \S(M)$, $e_i\otimes e_j^* \mapsto e_ie_j\eta.$
Note that $m$ is $\SU(2)$ equivariant and $\im(m)= \HH_{\eta}$. Using Proposition \ref{irreducible-decomposition}, we obtain $\ker(m)= \frsu(2) \oplus (\oplus_{k=1}^3 \s_k(\xi))$.
Moreover, $m(\rI)=-4\eta$ and $m(J_k)=-4\e_k\jota_k\eta$.

In addition, $R V_{\xi}\eta= J_1(V_{\xi})\eta$. Finally, 
 \[
\sum_{i=1}^4{\Theta_k(e_i)e_i}\jota_k\eta=\e_k J_k\Theta_k^\sharp\eta \quad \textrm{and} \quad \sum_{k=1}^3{\phi_k R \jota_k\eta} = \p_1 \eta - \p_2\jota_3\eta + \p_3\jota_2\eta\,.
 \]
\end{proof}

Next, we proceed to write the torsion in terms of the forms $(\alpha,\o_1,\o_2,\o_3)$ defined by a unit-length spinor $\eta \in \S(M)$ as in Lemma \ref{SU2-forms-spinors}.

\begin{proposition} \label{prop:cov-a-ok}
The covariant derivatives of the forms $(\alpha,\o_1,\o_2,\o_3)$ are governed by the formulae
\begin{align*}
(\onabla_Z \o_k)(X,Y) & =\e_k\la\nabla_Z \eta,(XY-YX)\jota_k\eta\ra, \quad k=1,2,3\,,\\
(\onabla_Z \a) (X) & = 2\la\nabla_Z\eta, X\jota_1\eta\ra\,,
\end{align*}
where $\onabla$ is the Levi-Civita connection and $\nabla$ is the spinorial connection.
\end{proposition}
\begin{proof}
 Take $X,Y,Z\in T_pM$ and extend them to vector fields with $\onabla X|_p=\onabla Y|_p=\onabla Z|_p=0$. Then,
\begin{align*}
(\onabla_Z \o_k)(X,Y) &= Z(\o_k(X,Y)) = \e_k \la \jota_k X \nabla_Z \eta, Y\eta\ra +\e_k \la \jota_kX\eta,Y\nabla_Z\eta\ra \\
& = \e_k\la \nabla_Z \eta,(XY-YX)\jota_k\eta\ra, \\
(\onabla_Z \a)(X) & = Z(\a(X))= -\la X\nabla_Z \eta, \jota_1\eta\ra-\la X\eta, \jota_1\nabla_Z\eta\ra\\
 & = 2\la\nabla_Z\eta, X\jota_1\eta\ra.
\end{align*}
\end{proof}

After computing the differentials, we prove a technical result:
\begin{lemma} \label{technical-Lemma}
For $X,Y \in \xi$, one has:
\begin{align*}
\o_1(S(X),Y) - \o_1(S(Y),X) & =  2( \m \o_1 -\l_3\o_2 + \l_2\o_3 + i(S_1)\omega_1)(X,Y)\,, \\
\o_2(S(X),Y) - \o_2(S(Y),X) & =  2( \l_3 \o_1 + \m \o_2 - \l_1\o_3 + i(S_2)\omega_2)(X,Y)\,, \\
\o_3(S(X),Y) - \o_3(S(Y),X) & =  2( -\l_2 \o_1 + \l_1 \o_2 + \m \o_3 + i(S_3)\omega_3)(X,Y)\,.
\end{align*} \end{lemma}
\begin{proof}
We prove the first equality, the others being similar. We analyze each irreducible part separately. 

Clearly $\o_1(\m X ,Y) - \o_1(Y,\m X)= 2\m \o_1(X,Y)$. Taking into account that $S_k J_1= \e_k J_1 S_k$, we have that $S_kJ_1$ is skew-symmetric for $k=1$ and symmetric for $k\in \{2,3\}$. Hence,
\[
\sum_{k=1}^3{g(J_1 S_k(X),Y) - g(J_1 S_k(Y),X)}= 2\omega_1(S_1(X),Y)\,.
\]

Finally we conclude:
\begin{align*}
\sum_{k=1}^3{\l_k g(J_1J_k(X),Y) - \l_kg(J_1J_k(Y), X)} & =  -2\l_3g(J_2(X), Y) + 2\l_2g(J_3(X), Y ) \\
 & =  2(-\l_3\o_2 + \l_2\o_3)(X,Y).
\end{align*}
Using that $S_0\in \frsu(2)$ we get, $\o_1(S_0(X),Y)+ \o_1(X,S_0(Y))=0$.
\end{proof}

\begin{proposition} \label{dalpha}
Let $\eta \in \S(M)$ be a unit-length spinor and let $\a$ be the 1-form of the $\SU(2)$ structure determined by $\eta$. Then
(with the notations of Proposition \ref{differentials-SU2}),
\[
d\a= \a \wedge \tau_1^4 + \sum_{k=1}^3 \tau_0^k\o_k + \tau_2^4\,,
\]
where:
\begin{enumerate}
\item[\textbullet] $\tau_0^1= -4\m$, $\tau_0^2=4\l_3$, $\tau_0^3 = -4 \l_2$,
\item[\textbullet] $\tau_1^4= 2J_1V_\xi^\sharp$,
\item[\textbullet] $\tau_2^4 = -4i(S_1)\omega_1$.
\end{enumerate}
\end{proposition}
\begin{proof}
Proposition \ref{prop:cov-a-ok} implies that $\frac{1}{2}d\a(X,Y)=\la\nabla_X \eta, Y\jota_1\eta\ra- \la\nabla_Y\eta, X\jota_1\eta\ra$. In order to compute $d\a|_\xi$ consider $X,Y\in \xi$; according to equation \eqref{eqn:1}, the orthogonal projection of $\nabla_{X}\eta$ to $\xi \eta$ is $S(X)\eta$. So that $\la\nabla_X \eta, Y\jota_1\eta\ra=\la S(X)\eta,J_1(Y)\eta \ra$.   Taking into account the previous observation, and Lemma \ref{technical-Lemma} we obtain:
\begin{align*}
\frac{1}{2}d\a(X,Y)& = \la X\eta,J_1S(Y)\eta\ra - \la Y\eta,J_1S(X)\eta\ra\\
& = -2(\m \o_1 -\l_3\o_2 + \l_2\o_3+ i(S_1)\omega_1)(X,Y).
\end{align*}

Finally, we compute $d\a(R,Y)$. Arguing as before, equation \eqref{eqn:1} implies that $\la \nabla_R \eta, \jota_1 Y \eta\ra=\la V_{\xi}\eta, \jota_1 Y \eta \ra$. In addition, $\la \nabla_Y\eta, j_1R\eta\ra=\la \nabla_Y\eta,\eta \ra=0$, according to Lemma \ref{su2-equalities}.  Thus,
$$
\frac{1}{2}d\a(R,Y)=\la V_\xi \eta, \jota_1Y\eta\ra- \la \jota_1 R \eta, \nabla_Y\eta\ra= \la V_\xi \eta, J_1(Y)\eta\ra.
$$
\end{proof}

\begin{proposition}\label{domegak}
Let $\eta \in \S(M)$ be a unit-length spinor and let $(\o_1,\o_2,\o_3)$ be the $2$-forms of the $\SU(2)$ structure determined by $\eta$. Then
\[
d\omega_k= \a \wedge \tau_2^k + \sum_{l=1}^3 \tau_0^{kl}\a\wedge \o_l + \tau_1^k \wedge \o_k\,,
\]
where:
\begin{enumerate}
\item[\textbullet]
$\tau_0^{kk}=4\l_1$, 
$\tau_0^{12}=4\l_2 + 2\p_3$,
$\tau_0^{13}=4\l_3 - 2\p_2$, 
$\tau_0^{23}=4\m - 2\p_1$,
\item[\textbullet] 
$\tau_1^k= -2\sum_{l\neq k}{\e_k J_l \Theta_l}$,
\item[\textbullet]
$\tau_2^1 = 4i(S_0)g$,
$\tau_2^2=4i(S_3)\o_3 $,
$\tau_2^3=-4i(S_2)\o_2$.
\end{enumerate}
\end{proposition}

\begin{proof}
Suppose that $X,Y,Z$ are orthonormal; then according to Proposition \ref{prop:cov-a-ok} we have $\onabla_Z\omega(X,Y)=2\e_k \la \nabla_Z \eta, XY \jota_k\eta\ra$, thus:
\begin{align} \label{eqn:3}
\e_k \frac{1}{2}d\omega_k(X,Y,Z) &= \la \nabla_X \eta, YZ \jota_k\eta\ra  - \la\nabla_Y \eta, XZ \jota_k\eta\ra + \la\nabla_Z \eta, XY \jota_k\eta\ra.
\end{align}
We first assume that $X,Y,Z\in \xi$. Then,
\begin{align*}
\e_k \frac{1}{2}d\omega_k(X,Y,Z) &= \la X\nabla_X\eta + Y\nabla_Y \eta + Z\nabla_Z \eta, XYZ\jota_k \eta \ra.
\end{align*}
Suppose in addition that $W\in \xi$ has length one, it is orthogonal to $\la X,Y,Z\ra$ and that the orthonormal frame $(X,Y,Z,W,R)$ is positively oriented, then
\begin{enumerate}
\item $X\nabla_X\eta + Y\nabla_Y \eta + Z\nabla_Z \eta = \sD\eta - W \nabla_W \eta - R \nabla_R \eta $,
\item The positively-oriented unit-length volume form is $\nu= X^*\wedge Y^* \wedge Z^* \wedge W^* \wedge R^*$. From the equality $\nu \eta = - \jota_1\eta= R\eta$ (see Lemma \ref{su2-equalities} (2) and (4)) we obtain
$XYZW\eta=\eta$ and thus, $XYZ\eta=-W\eta$. Therefore,
$$
XYZ\jota_k\eta=\e_k \jota_k XYZ\eta= -\e_k \jota_k W \eta= -\e_k J_k(W)\eta.
$$
\end{enumerate}
Therefore,
\[
\frac{1}{2}d\omega_k(X,Y,Z)= -\la\sD\eta,J_kW\eta\ra + \la W\nabla_W\eta, J_kW\eta\ra + \la R\nabla_{R} \eta, J_kW \eta\ra\,.
\]
From Proposition \ref{Dirac-su2-spinors} we obtain that the orthogonal projection of $-\sD\eta$ to $\xi \eta$ is
$
 (-J_1(V_{\xi} + \Theta_1^\sharp)+ J_2(\Theta_2^\sharp)+J_3(\Theta_3^\sharp)) \eta.
$
Since $J_l(\a^\sharp)^*=-J_l(\a)$ if $\a \in \xi^*$ we have:
$$
-\la\sD\eta, J_k(W)\eta\ra= (-J_1(V_{\xi})^* + \sum_{l=1}^3 {\e_l J_l(\Theta_l)}) (J_k(W)).
$$
Morever, $\la W\nabla_W\eta, J_k(W)\eta\ra=\e_k\la\nabla_W\eta,\jota_k\eta \ra =\e_k \Theta_k(W)$ according to equation \eqref{eqn:1}. Taking into account the same equation and the fact that the spinor $R \jota_k\eta=-\e_k\jota_k\jota_1\eta$ is perpendicular to $\xi\eta$, we obtain 
$
\la R\nabla_{R\eta} \eta, J_kW \eta\ra= \la J_1 V_{\xi} \eta, J_k W \eta\ra= (J_1V_{\xi})^*(J_kW)\,.
$

From the previous discussion, we deduce: 
\[
\frac{1}{2}d\omega_k(X,Y,Z)=  \sum_{l=1}^{3} \e_l (J_l\Theta_l)(J_kW) + \e_k \Theta_k(W) = \sum_{l\neq k} \e_l J_l\Theta_l(J_kW)\,.
\]

The previous equality implies that  $\star_\xi (\t_1^k \wedge \o_k)= 2 \sum_{l\neq k} \e_l J_k (J_l\Theta_l)$, since $(X,Y,Z,W)$ is a positive frame. Taking into account Lemma \ref{hodge-star}, we obtain $\tau_1^k=  -2 \sum_{l\neq k} \e_l J_l\Theta_l$.

Suppose that $X,Y \in \xi$ are orthonormal vectors; we now compute $i(R)d\omega$ by using equation \eqref{eqn:3}.
To arrange the second and the third summands of equation \eqref{eqn:3}, we observe that if $Z\in \xi$, then:
$$
 \a Z \jota_k\eta = \a \e_k J_k(Z)\eta = \e_k J_k(Z) \jota_1\eta = \e_k(J_1(J_k(Z)))\eta.
 $$  
 Thus,
\[
\frac{1}{2}d\omega_1(R,X,Y)= \e_k\la\nabla_{R} \eta, XY \jota_k\eta\ra - \la S(X) \eta, J_1(J_k(Y)) \eta\ra + \la S(Y)\eta,J_1(J_k(X)) \eta\ra.
\]

We first deal with the summand $\e_k\la\nabla_{R} \eta, XY \jota_k\eta\ra$. According to equation \eqref{eqn:1} we have:  $\la \nabla_R \eta, XY \jota_k\eta\ra = \la V_{\xi}\eta, XY \jota_k\eta\ra+ \sum_{l=1}^3 \p_l\la \jota_l\eta,XY \jota_k \eta\ra$. Due to Remark \ref{orthogonal-clifford-product}, $\la V_\xi \eta, XY\jota_k \eta\ra=\la -J_k(V_\xi)\eta,XY\eta\ra=0$. We now observe that $\la \jota_l\eta,XY \jota_k \eta\ra=\e_k \e_l \la J_k(J_l(X))\eta,Y\eta\ra$ and we compute:
\begin{align*}
\e_1 \la\nabla_R \eta, XY \jota_1\eta\ra & =+ \p_3 \o_2 - \p_2\o_3\,, \\
\e_2 \la\nabla_R \eta, XY \jota_2\eta\ra & =  -\p_3 \o_1 - \p_1 \o_3\,, \\
\e_3 \la\nabla_R \eta, XY \jota_3\eta\ra & = +\p_2 \o_1 +\p_1\o_2\,.\end{align*} 
We now deal the summand $T^k(X,Y)=- \la S(X) \eta, J_1(J_k(Y)) \eta\ra + \la S(Y)\eta,J_1(J_k(X)) \eta\ra$. 
From Definition \ref{endomorphism-S}, one can check:
\begin{align*}
T^1(X,Y)=&2\la S_0(X)\eta,Y\eta\ra + 2\sum_{k=1}^3\l_k \la J_k(X)\eta,Y\eta\ra \\
=& 2(i(S_0)g +\l_1\o_1+ \l_2 \o_2 + \l_3 \o_3)(X,Y) .
\end{align*}
In addition, $T^2(X,Y)=\omega_3(S(X),Y)-\omega_3(S(Y),X)$ and  $T^3(X,Y)=-(\omega_2(S(X),Y)- \omega_3(S(Y),X))$. Taking into account Lemma \ref{technical-Lemma} we obtain:
\begin{align*}
T^2(X,Y)&=  2(-\l_2 \o_1 + \l_1 \o_2 + \m \o_3 + i(S_3)\omega_3)(X,Y)\, , \\
T^3(X,Y)&= 2(-\l_3 \o_1 - \m \o_2 + \l_1\o_3 - i(S_2)\omega_2)(X,Y)\, .
\end{align*}
In sum,
$i(R) d \o_1=4 i(S_0)g + 4 \l_1\o_1+ (4 \l_2 + 2 \p_3)\o_2 + (4\l_3-2\p_2 \o_3)$.
Thus, $\t_0^{kk}=4 \l_1$, $\t_0^{12}=4 \l_2 + 2 \p_3$, $\t_0^{13}=4\l_3-2\p_2 $ and $\t_2^0=4 i(S_0)g$. 
The remaining equalities are obtained similarly.

\end{proof}

The previous results allow us to write the equations for $\SU(2)$ structures induced by a harmonic spinor. 
We equate $\sD\eta=0$ in Proposition \ref{Dirac-su2-spinors}, and use the values of $d\a$ and $d\o_k$ computed in
Propositions \ref{dalpha} and \ref{domegak}. Rewriting with the notations of Proposition \ref{differentials-SU2}, we get:

\begin{corollary}
The spinor $\eta$ is harmonic if and only if $\SU(2)$ structure determined by $\eta$, $(\a,\o_1,\o_2,\o_3)$, verifies:
\begin{align*}
d\a& =  + \tau_0^{23}\o_1 + \tau_0^{13}\o_2 - \tau_0^{12}\o_3  + \frac{1}{2} \sum_{k=1}^3(\a \wedge \tau_1^k) + \tau_2^4\,, \\
d\o_1 & = + \tau_0^{12}\a\wedge\o_2 + \tau_0^{13}\a \wedge \o_3  + \tau_1^1 \wedge \o_1 + \a \wedge \tau_2^1 \,,\\
d\o_2 &= - \tau_{0}^{12}\a \wedge \o_1 + \tau_0^{23}\a \wedge \o_3 + \tau_1^2 \wedge \o_2 + \a \wedge \tau_2^2  \,,\\
d\o_3 & =  - \tau_0^{13}\a \wedge \o_1 - \tau_0^{23} \a \wedge \o_2  + \tau_1^3 \wedge \o_3 + \a \wedge \tau_2^3\,.
\end{align*}
\end{corollary}
\begin{proof}
We equate $\sD\eta=0$ in Proposition \ref{Dirac-su2-spinors}, and we obtain $4\m=\p_1$, $\l_1=0$, $4\l_2=-\p_3$, $4\l_3=\p_2$, and $-J_1(V_1^*)= \sum_{k=1}^3 \e_k J_k(\Theta_k)$.
According to Propositions \ref{dalpha} and \ref{domegak}, the  $0$-forms are related as follows:
\begin{align*}
\t_0^{kk}=&4\l_1=0,\\
\t_0^{12}=& 4 \l_2 + 2 \p_3= -4 \l_2=-\t_0^3,\\
\t_0^{13}=& 4 \l_3- 2 \p_2 =-4 \l_3 = \t_0^2,\\
\t_0^{23}=& 4 \m - 2\p_1 = - 4\m=\t_0^1.
\end{align*}
In addition, $\t_1^4= 2J_1(V_\xi^*)=-2\sum_{k=1}^3 \e_k J_k(\Theta_k)= \frac{1}{2} \sum_{k=1}^3 \t_1^k$. 
\end{proof}

In \cite[Definition 1.5]{CS07} the authors defined {\em hypo} $\SU(2)$ structures as those verifying 
\[
 d\o_1=0 \quad \textrm{and} \quad  d(\a \wedge \o_k)=0, \quad k=2,3\,.
\]

The intersection between hypo and spin-harmonic stuctures is characterized by the equations:

\begin{minipage}[t]{0.4\textwidth}
\begin{itemize}
 \item $d\a = - \tau_0^{23}\o_1 + \tau_2^4$;
 \item $d\o_1  = 0$;
\end{itemize}
\end{minipage}
\begin{minipage}[t]{0.56\textwidth}
\begin{itemize}
 \item $d\o_2 = + \tau_0^{23}\a \wedge \o_3 + \a \wedge \tau_2^2$;
 \item $d\o_3 =   - \tau_0^{23} \a \wedge \o_2  + \a \wedge \tau_2^3$.
\end{itemize}
\end{minipage}

In section \ref{sec:harm-nilm} we present three nilmanifolds that admit $\SU(2)$ invariant structures in this intersection.



\section{Dirac operator of invariant spinors on Lie groups}\label{sec:Dirac-Operator}

\subsection{Spin structures on Lie groups}\label{subsec:Spin-Structures}

Let $(\rG,g)$ be an $n$-dimensional connected, simply connected Lie group endowed with a left-invariant metric. Fix an orthonormal left-invariant frame $(e_1,\dots,e_n)$; the frame bundle of $\rG$ is $\PSO(\rG)=\rG\times \SO(n)$ and its unique spin structure is $\PSpin(\rG)=\rG\times \Spin(n)$. Fix also an irreducible representation  $\rho \colon \Cl_n \to \End_\bk(W)$. The spinor bundle of $\rG$ is $\S(\rG)=\rG\times W$ and the Clifford multiplication by a vector field $X(x)=\sum_{i=1}^n{X^i(x)e_i(x)}$ is given by $X(x)\p(x)=\sum_{i=1}^n X^i(x)\rho(e_i)\p(x)$ where $\{e_i \}_{i=1}^n$ is the canonical basis of $\RR^{n}$. Each spinor is identified with a map $\p \colon \rG \to W$ and we call the spinor $\p$ {\em left-invariant} if it is constant.

Let $\Gamma$ be a discrete subgroup of $\rG$ and 
$\pi \colon \rG \to \Gamma \backslash \rG$ be the canonical projection. We endow $\Gamma \backslash \rG$ with the metric, also denoted $g$, which pulls back to $g$ under $\pi$.

\begin{lemma} 
There is a bijective correspondence between homomorphisms $\e\colon \Gamma \to \{\pm 1\}$ and spin structures on $\Gamma \backslash \rG$:
$$
\e \longmapsto \PSpin(\Gamma \backslash \rG)^\e= \Gamma \backslash (\rG\times \Spin(n))\,,
$$ where the action is $y \cdot (x,\tilde{h}) = (yx,\e(y)\tilde{h})$, for $y\in \Gamma$.
\end{lemma}

\begin{proof}
Spin structures on $\Gamma \backslash \rG$ are in a bijective correspondence with liftings of the action $\PSO(\rG) \times \Gamma \to \PSO(\rG), y\cdot 
F_x = d (\rL_{y})_x(F_x)$ where $\rL_y$ denotes the left multiplication by $y$ (see \cite[page 43]{F00}). This action commutes with action of $\SO(n)$ on $\PSO(\rG)$ and therefore a lifting of this action commutes with the action of $\Spin(n)$ on $\PSpin(\rG)$.

According to the identification $\PSO(\rG)= \rG \times \SO(n)$ given by $(e_1,\dots,e_n)$, the action is
$y \cdot (x,h) = (yx,h)$.
A lifting of the action to $\PSpin(\rG)=\rG\times \Spin(n)$ must verify $y\cdot (x,1)
=(yx,\e(y)1)$ for a some map $\e \colon \Gamma \to \{\pm 1\}$, which is necessarily a homomorphism. 
The previous discussion shows that this property determines the action. 
\end{proof}


The spinor bundle associated to $\PSpin(\Gamma \backslash \rG)^\e$ is $\S(\Gamma \backslash \rG)^\e=
\PSpin(\Gamma \backslash \rG)^\e \times_{\rho} W$, which is isomorphic to $ \Gamma \backslash (\rG\times W)$ via the induced action $y\cdot (x,v)=(yx,\e(y)v)$. 
Spinors are then identified with maps $\p \colon \rG \to W$ such that $\p(yx)=\e(y)\p(x)$ for $x\in \rG$, 
$y \in \Gamma$, and Clifford multiplication of a spinor $\p \colon \rG \to W$ with a vector field $X \in \frX(\Gamma \backslash \rG)$ with $X(\pi(x))= \sum_{i=1}^n{X^i(x)d\pi_{x}(e_i(x))}$ is given by
$X\p(x) = \sum_{i=1}^n{X^i(x) \rho(e_i)\p(x)}$. Moreover, a spinor $\p \in \S(\Gamma \backslash \rG)^\e$ lifts to a unique spinor $\bar{\p}\in \S(\rG)$ and both are identified with the same map $\rG \to W$. Using this identification, for a left-invariant vector field $X\in \frX(\rG)$ we have $\nabla_{d\pi_x(X)} \p(x)= \nabla_X \bar{\p}(x)$ and, according to \cite[page 60]{F00},
\begin{equation} \label{cov}
\nabla_X \bar{\p} = d_{X} \bar{\p} + \frac{1}{2}\sum_{j<k} g(\nabla_{X}e_j,e_k)e_je_k \bar{\p}.
\end{equation}

In the sequel we focus on a quotient $\Gamma\backslash G$ and on spinors that lift to left-invariant spinors on $G$; we call those {\em left-invariant spinors}. Of course, they are associated to the trivial spin structure and they are constant. Special examples are given by {\em nilmanifolds}, where $G$ is nilpotent, and {\em solvmanifolds}, where $G$ is solvable.

In particular, we restrict our attention to left-invariant {\em harmonic} spinors. Mind that the non existence of left-invariant harmonic spinors does not imply the non existence of harmonic spinors associated to the trivial spin structure. For instance, from Proposition \ref{invariant-Dirac-operator} one can deduce that a $3$-dimensional nilmanifold, quotient of the Heisenberg group, does not admit left-invariant harmonic spinors; however, Corollary 3.2 in \cite{Ammann-Baer} implies that every spin structure on such a nilmanifold admits a left-invariant metric with non-zero harmonic spinors.

\subsection{Dirac operator}\label{subsec:Dirac Operator}

Let $(\rG,g)$ be a Lie group endowed with a left-invariant metric, let $(e_1,\ldots,e_n)$ be a left-invariant orthonormal frame with dual coframe $(e^1,\ldots,e^n)$. Let $\Gamma$ be a discrete subgroup of $\rG$ and consider the spin structure associated to the trivial action on $\Gamma \backslash \rG$. We  follow the notation of the previous subsection. 

\begin{proposition}\label{invariant-Dirac-operator}
Let $\p$ be a left-invariant spinor. Then
\begin{equation}\label{eq:invariant-Dirac-operator}
4 \sD \p  = -\sum_{i=1}^n{(e^i \wedge de^i + i(e_i)de^i)}\p \,. 
\end{equation}
\end{proposition}

\begin{proof}
First we compute the covariant derivative of $\p$ according to formula (\ref{cov}). 
Note that $d_{e_i}\p=0$ because $\p$ is left-invariant. We use Koszul formula to obtain
\[
 2\onabla_{e_i} e_j=(i(e_i)de^j + i(e_j)de^i)^\sharp - \sum_{k} de^k(e_i,e_j)e_k\, ,
\]
where $\onabla$ is the Levi-Civita connection and $\nabla$ is the spinor connection.
Therefore,
\begin{align*}
\nabla_{e_i} \p & = \frac{1}{4} \left( \sum_{j<k}{\left(de^j(e_i,e_k) + de^i(e_j,e_k) - de^k(e_i,e_j)\right)e_je_k} \right) \p 
\\
& =  \frac{1}{4} \left(de^i\p - 2\sum_{j,k}{de^k(e_i,e_j)e_je_k}\p + 2\sum_{k}{de^k(e_k,e_i)}\right)\p\,.
\end{align*}
From this we get:
\begin{align*}
4\sD \p& = \sum_{i=1}^n{e^i de^i}\p - 2\sum_{i<j,k}{de^k(e_i,e_j)e_ie_je_k}\p + 2\sum_{i,k}{de^k(e_k,e_i)e_i}\p \\
& = \sum_{i=1}^n{(e^i de^i - 2 de^i e_i + 2 i(e_i)de^i)}\p = -\sum_{i=1}^n{( e^i \wedge de^i + i(e_i)de^i)}\p\,,
\end{align*}
where we have used that $e^i de^i\p = (e^i\wedge de^i - i(e_i)de^i)\p$ and $(de^i) e^i\p = (e^i\wedge de^i + i(e_i)de^i)\p$.
\end{proof}

Since our focus is on nilmanifolds and solvmanifolds, we specialize Proposition \ref{invariant-Dirac-operator} to this setting. Recall that a frame $(e_1,\dots, e_n)$ of a nilpotent Lie group is called \emph{nilpotent} if
\[
 [e_i,e_j]=\sum_{k>i,j}c_{ij}^ke_k\,.
\]

\begin{corollary} \label{nilpotent-Dirac-operator}
Let $\rG$ be a nilpotent Lie group and let $(e_1,\dots, e_n)$ be an orthonormal nilpotent frame. Let $\p \colon \rG \to W$ be a left-invariant spinor; then 
\begin{equation}\label{eq:Dirac-nilpotent}
4\sD\p = -\sum_{i=1}^n{(e^i \wedge de^i)}\p\,.  
\end{equation}
In particular, the operator $\sD$ is $\la\cdot,\cdot\ra$-symmetric on the space of invariant spinors.
\end{corollary}

Next, suppose that $\frg$ is a rank-1 extension of a nilpotent Lie algebra $\frn$, and let $\rG$ and $\rN$ be the associated simply connected Lie groups. As vector spaces $\frg= \la e_0 \ra \oplus \frn$; the Lie bracket in $\frg$ is given by
\[
 [e_0,X]_\frg=\cDD(X), \quad [X,Y]_\frg=[X,Y]_\frn \quad \textrm{for} \ X,Y\in\frn\,,
\]
where $\cDD\colon\frn\to\frn$ is a derivation. In terms of covectors, $\cDD$ can be seen as a linear map $\frn^*\to\frn^*$ such that $d_\frn \circ \cDD=\cDD \circ d_\frn$, where $d_\frn\colon\L^k\frn^*\to \L^{k+1}\frn^*$ is the Chevalley-Eilenberg differential. Extending $\alpha\in\L^k\frn^*$ by zero to $\la e_0\ra$, one has 
\begin{equation}\label{eqn:2}
d_{\frg}\a= d_{\frn}\a + (-1)^{k+1} \cDD(\a)\wedge e^0\,,
 \end{equation} 
where $d_\frg\colon\L^k\frg^*\to \L^{k+1}\frg^*$ is the Chevalley-Eilenberg differential. We also suppose that $\rG$ is endowed with an invariant metric which makes $e^0$ orthogonal to $\frn^*$.

\begin{corollary}\label{solvable-Dirac-operator}
Suppose that $(e_1,\dots, e_n)$ is an orthonormal frame of $\rN$ and let $\p \colon \rG \to W$ be a left-invariant spinor. Then 
\begin{equation}\label{eq:Dirac-rank1}
4 \sD\p = -\sum_{i=1}^n{(e^i \wedge d_{\frn}e^i+ i(e_i)d_{\frn}e^i + e^0 \wedge e^i \wedge \cDD(e^i))}\p - \tr(\cDD)e^0\p \,.  
\end{equation}
In particular if $\cDD$ is symmetric and $(e_1,\dots,e_n)$ is a basis of eigenvectors then $4\sD\p = -\sum_{i=1}^n{(e^i \wedge d_{\frn}e^i) + i(e_i)d_{\frn}e^i}\p - \tr(\cDD)e^0\p$ .
\end{corollary}

\begin{proof} 
The formula is deduced from Proposition \ref{invariant-Dirac-operator} and (\ref{eqn:2}). 
In addition, if $\cDD$ is symmetric and $(e^1,\dots, e^n)$ is a basis of eigenvectors of $\cDD$, then $e^i\wedge \cDD(e^i)=0$.
\end{proof}

\subsection{The operator \texorpdfstring{$\sD{}^2$}{Lg} on nilmanifolds}\label{subsec:D^2}     %

The square of the Dirac operator is an elliptic operator with positive eigenvalues. In this subsection we fix the trivial spin structure on a nilmanifold $\Gamma \backslash \rG$ associated to the trivial action and obtain
a formula for the square of the Dirac operator over the space of left-invariant spinors. This will allow us to understand the eigenvalues of the $5$-dimensional Dirac operator in Section \ref{sec:harm-nilm}. A straightforward computation gives the following result:

\begin{lemma} \label{square-formula}
Suppose $(e_1,\dots,e_n)$ is an orthonormal nilpotent frame of $\rG$ and $\p \colon \rG \to W$ a left-invariant spinor, then:
\begin{equation}\label{eq:square-Dirac}
16 \sD{}^2\p=\left(\sum_i -(de^i)^2 + \sum_{i < j}{(e^{ij} de^ide^j - de^jde^ie^{ij})} \right)\p \,. 
\end{equation}
\end{lemma}

We discuss each summand of \eqref{eq:square-Dirac}. We use the juxtaposition of indices to denote Clifford products, for instance $e_{ij}=e_i e_j$. Moreover, each $\b = \sum_{i_1<\dots <i_k} \b_{i_1,\dots, i_k}e^{i_1\dots i_k} \in \L^k \frg^*$ is identified with the element $\sum_{i_1<\dots <i_k} \b_{i_1,\dots, i_k}e_{i_1\dots i_k}$ of the Clifford algebra. This identification does not depend on the orthonormal basis chosen. We also set
\[
\g_{ij} = e^{ij}de^ide^j - de^jde^i e^{ij}\,. 
\]

\begin{lemma} \label{dx^2}
Take $\o$ in $\L^2 \frg^*$. Using the previous identifications,
\[
\o \cdot \o = - \| \o \|^2 +  \o \wedge \o\,.
\]
\end{lemma}

\begin{proof}
Let $(e_1,\dots,e_n)$ be an orthonormal basis and write $\o = \sum_{i<j}\o_{ij}e_{ij}$. If $i,j,k,l$ are distinct indices, 
then it is easy to obtain that $e_{ij}e_{ik} + e_{ik}e_{ij}=0$ and that $e_{ijkl} + e_{klij}= 2e_{ijkl}$. A combination of these properties leads to the equality:
\[
\left( \sum_{i<j}{\o_{ij} e_{ij}} \right)^2 = -\sum_{i<j}{\o_{ij}^2}  +  2\sum_{i<j<k<l}{(\o_{ij}\o_{kl} + \o_{il}\o_{jk} - \o_{ik}\o_{jl})}e_{ijkl}\,,
\]
which proves the lemma.
\end{proof}

\begin{remark} \label{real-structure}
The operator $e_{ijkl} \cdot $ verifies $(e_{ijkl}\cdot )^2=\rI$ and it is not an homotethy.
Let $\D_{\pm}$ be the eigenspace of $\S(G)$ associated to $\pm 1$ and take $\p_{\pm} \in \Delta_{\pm}$. Then,
\[
(\o_{ij} e^{ij} + \o_{kl} e^{kl})^2\p_\pm = -(\o_{ij} \mp \o_{kl})^2\p_\pm\,.
\]
This endomorphism is invertible except when $\o_{ij}=\pm \o_{kl}$; in this case the kernel is $\Delta_\pm$.
\end{remark}

\begin{lemma} \label{dxidxj} 
Let $(e_1,\dots,e_n)$ be an orthonormal nilpotent frame of $\frg$ and $i<j$. Then
\[ 
\g_{ij}= - 2 de^i \wedge i(e_i)de^j \wedge e^j + 2 \sum_{k<i}{i(e_k)de^i \wedge i(e_k)(de^j|_{\la e_i \ra^\perp}) \wedge e^{ij}}\,.
\]
\end{lemma}

\begin{proof}
We denote $\a= i(e_i)de^j \in \frg^*$ and $\b= de^j|_{\la e_i\ra}^\perp \in \L^2 \la e^i \ra^\perp$, 
that is, $de^j= e^i\wedge \alpha + \beta$. In this notation, we observe that
$ e_{ij}de^i de^j= e_{ij}de^i(e^i\wedge \a + \b)= de^i(- e^i\wedge \a + \b )e_{ij}$
and that $e_i\beta = \beta e_i$. Hence,
\[
 \g_{ij}= \left(de^i (-e^i\wedge \a + \b)- (e^i\wedge \a + \b)de^i \right) e_{ij} 
             =  - (de^i\a + \a de^i)e_j + (de^i\beta - \beta de^i) e_{ij}.
\]
We now identify the terms in the summand. On the one hand, if we write $de^i=\a\wedge \a' + \beta'$ where $\a'=i(\a^\sharp)de^i$ and $\b'=de^i|_{\la \a^\sharp \ra^\perp}$, we obtain:
\[ 
(de^i \a + \a de^i)e_j = 2 (\beta' \a) e^j = 2de^i \wedge \a \wedge e^j \, .
 \]
On the other hand, it is sufficient to prove $(de^i\beta - \beta de^i)= 2\sum_{k<i}i(e_k)de^i\wedge i(e_k)\beta$ 
in the case that $de^{i}=e^{pq}$ and $\b=e^{l m}$ with $l< m$ and $p<q$. We distinguish two cases:
\begin{enumerate}
\item If $(p,q)=(l,m)$ or $p,q \notin \{ l, m\}$, then $e^{pq}e^{l m}-e^{l m}e^{pq}=0$. 
In addition, we have $\sum_{k=1}^{j-1}{i(e_k)e^{pq}\wedge i(e_k)e^{l m}}=0$.
\item In other case; for instance if $p=l$ and $q\neq m$, then $e^{pq}e^{pm}- e^{pm}e^{pq}=2e^{qm}$ and $2\sum_{k=1}^{j-1}{i(e_k)e^{pq}\wedge i(e_k)e^{pm}}=2e^{qm}$. The other instances are similar.
\end{enumerate}

\end{proof}

From this we obtain:

\begin{corollary} \label{4-form}
Let $(e_1,\dots,e_n)$ be a nilpotent orthonormal frame of $\frg$ and let $\p$ be a left-invariant spinor; then,
\begin{align*}
16 \sD{}^2 \p = & \sum_{i=1}^n{(\|de^i \|^2 - de^i \wedge de^i)}\p  
          - 2  \sum_{i<j}{(de^i \wedge i(e_i)de^j \wedge e^j)}\p  \\
        + & \, 2 \sum_{k<i<j} {i(e_k)de^i \wedge i(e_k)(de^j|_{\la e_i \ra^\perp})\wedge e^{ij}}\p .
\end{align*}
\end{corollary}


\section{Spin-harmonic structures on nilmanifolds}\label{sec:harm-nilm}

In order to determine left-invariant harmonic structures on nilmanifolds one has to compute the Dirac operator associated to each left-invariant metric and study its kernel. 
In dimension $4$ and $5$ we give a list of all left-invariant metrics and compute the eigenvalues of the Dirac operator by means of the metric using Corollary \ref{4-form}. 
We will also give a list of $6$-dimensional nilmanifolds that admit left-harmonic structures and list one such metric on each algebra.

Note that the existence of left-invariant harmonic spinors on a nilmanifold $\Gamma \backslash \rG$ depends on the Lie algebra $\frg$. For this reason, we sometimes write that the Lie algebra $\frg$ admits left-invariant harmonic spinors.

For Lie algebras we use Salamon's notation: $(0,0,12,13)$ denotes the 4-dimensional Lie algebra with basis $(e_1,e_2,e_3,e_4)$ and dual basis $(e^1,e^2,e^3,e^4)$, with differential $de^1=de^2=0$, $de^3=e^{12}$ and $de^4=e^{13}$. The list of nilmanifolds can be found in \cite{BM2}.

\subsection{4-dimensional nilmanifolds}\label{subsec:4-nilm}

In terms of an orthonormal nilpotent basis, a list of non-abelian $4$-dimensional metric nilpotent Lie algebras is:

\begin{center}
{\tabulinesep=1.2mm
\begin{tabu}{|l|c|c|c|}
\hline
                    &                & $de^3$ & $de^4$ \\
\hline
$\rL_3\oplus \rA_1$ & $(0,0,0,12) $ & $0$ & $\m_{12}e^{12}$  \\ \hline
$\rL_4 $            & $(0,0,12,13)$ & $\m_{12}e^{12}$ & $e^1(\lambda_{12}e^2 + \m_{13}e^{3})$ \\ \hline
\end{tabu}}
\end{center}
Here $\mu_{ij}$ denote structure constants which are necessarily non-zero, while $\l_{ij}$ may vanish.

\begin{theorem}\label{4-harmonic}
$4$-dimensional non-abelian nilmanifolds have no left-invariant harmonic spinors. 
\end{theorem}

\begin{proof}
The Dirac operator on $\rL_3\oplus \rA_1$ is
$
\sD\p=\m_{12}e^{124}\p,
$
and the square of the Dirac operator on $\rL_4$ is
$
16 \sD{}^2\p = (\m_{12}^2 + \m_{13}^2 + \l_{12}^2 )\p.
$
Both are invertible.
\end{proof}

\subsection{5-dimensional nilmanifolds}\label{subsec:5-nilm}

As in Section \ref{subsec:spinors}, we fix an irreducible representation of $\Cl_5$, $\rho_5 \colon \Cl_5 \to \End_\CC(W)$, with complex structure $\jota_1=\rho_5(\nu_5)$ and a quaternionic stucture $\jota_2$ that anticommutes with the Clifford product; define $\jota_3=\jota_1\circ \jota_2$. For instance, let $\rho_6$ be the representation of the real $6$-dimensional Clifford algebra described on subsection \ref{subsec:spinors} and define $\rho_5=\rho_6\circ i_5$, as in Proposition \ref{restriction}. Then, $\jota_1=\rho_5(\nu_5)$ and $\jota_2=\rho_6(e_6)$.

We first use Corollary \ref{4-form} to obtain the eigenvalues of the Dirac operator. In the presence of a harmonic spinor $\eta$, we can relate the operator $16 \sD{}^2$ with the 1-form $\a$ of the $\SU(2)$ structure defined by $\eta$.

\begin{proposition} \label{5-square}
Let $(e_1,\dots,e_5)$ be an orthonormal nilpotent basis of $\frg$ and let $\p$ be an invariant spinor. Then
$16 \sD{}^2\p= \m \p + v \jota_1\p$
where $\m = \sum {\|de^i\|^2}$ and 
\begin{align*}
 v^\sharp  =& \star(de^5\wedge de^5) + 2  \star\left(\sum_{i=3}^4de^i \wedge i(e_i)de^5 \wedge e^5\right) \\
     & - 2 \star\left(\sum_{i=3}^4\sum_{k=1}^3 {i(e_k)de^i \wedge i(e_k)(de^5|_{\la e_i \ra^\perp})\wedge e^{i5}  }\right) .
\end{align*}
In addition, $\mu\geq \|v\|$ and the restriction of the operator $4 \sD$ to the space of invariant spinors has four complex eigenspaces, associated to $\pm (\m \pm \|v\|)^{\frac{1}{2}}$. The endomorphism $\jota_2$ maps the eigenspace associated to $(\m \pm \|v\|)^{\frac{1}{2}}$ to the eigenspace associated to $-(\m \pm \|v\|)^{\frac{1}{2}}$.
In particular, there exist left-invariant harmonic spinors if and only if $\m = \|v\|$.
\end{proposition}

\begin{proof}
First observe that if $\gamma \in \L^4\frg^*$, then $\gamma \p = -(\star\gamma) \jota_1 \p$. This computation is straightforward for simple forms and is extended to $\L^4 \frg^*$ by linearity. Note also that the nilpotency property guarantees that $de^j\wedge de^j=0$ for $j\leq 4$ and that $\gamma_{34}=0$.
Those remarks and Corollary \ref{4-form} allow 
us to conclude the first statement. From this we get that the eigenvalues of $16\sD{}^2$ are $\m \pm \|v\|\geq 0$ and the eigenvalues of $4\sD$ are therefore, $\pm (\m \pm \|v\|)^{\frac{1}{2}}$.
Finally, the equality $\nabla_X \jota_k\p=\jota_k\nabla_X\p$, implies $\sD \jota_k=\e_k\jota_k \sD$ which is sufficient to conclude the rest.
\end{proof}


\begin{proposition} \label{alpha-v}
Let $(\a,\o_1,\o_2,\o_3)$ be the $\SU(2)$ structure determined by a left-invariant unit-length spinor $\eta$. Let $(e_1,\dots,e_5)$ be an  orthonormal nilpotent frame and consider $\mu$ and $v$ defined as in Proposition \ref{5-square}.
The spinor $\eta$ is harmonic if and only if $\|v\|= \m$ and $v=-\m\a^{\sharp}$. 
\end{proposition}

\begin{proof}
Decompose $v= \l \a^{\sharp}+ w$ according to the orthogonal decomposition $\la \a^\sharp \ra \oplus \xi$. 
By Corollary \ref{4-form}, $\sD{}^2\eta = \m \eta + (\l\a^\sharp + w)\jota_1\eta = (\m+ \l)  \eta + w\jota_1\eta$, using that 
$\a^\sharp \jota_1 \eta = \jota_1 \a^\sharp \eta = \jota_1 (-\jota_1 \eta ) = \eta$, from Lemma \ref{su2-equalities}(2).
This implies, according to Lemma \ref{decomposition-2}, that  $w=0$ and $\mu=-\l$.
Thus, $v= -\m \a^{\sharp}$.
\end{proof}

{}From these results we observe that on a nilpotent Lie algebra, the component of $v$ on the subspace $\la e^5 \ra$ depends on the non-degeneracy of $de^5$. Moreover, taking into account the structure equations of $5$-dimensional nilpotent Lie algebras given in Lemma \ref{5-classification}, one deduces that the component of $v$ on $\la e^4 \ra$ is always $0$. In any case, the vector $v$ is going to be determined in Theorem \ref{5-harmonic}.

The non-abelian nilpotent 5-dimensional Lie algebras are the following:

\noindent\begin{minipage}{0.49\textwidth}
\begin{itemize}
\item $\rL_3\oplus \rA_2$, $(0,0,0,0,12)$
\item $\rL_4\oplus \rA_1$, $(0,0,0,12,14)$
\item $\rL_{5,1}$, $(0,0,0,0,12+34)$
\item $\rL_{5,2}$, $(0,0,0,12,13)$
\end{itemize}
\end{minipage}
\begin{minipage}{0.49\textwidth}
\begin{itemize}
\item $\rL_{5,3}$, $(0,0,0,12,14+23)$
\item $\rL_{5,5}$, $(0,0,12,13,23)$
\item $\rL_{5,4}$, $(0,0,12,13,14)$
\item $\rL_{5,6}$, $(0,0,12,13,14+23)$
\end{itemize}
\end{minipage}


\begin{lemma} \label{5-classification}
The following table contains a list of non-abelian $5$-dimensional metric nilpotent Lie algebras in terms of an orthonormal nilpotent basis $(e_1,\ldots,e_5)$ with dual basis $(e^1,\ldots,e^5)$. Here 
$\mu_{ij}$ denote structure constants which are non-zero, while $\l_{ij}$ or $\l_{ij;k}$ denote those which may be zero.
\begin{center}
\resizebox{14.75cm}{!}
{\tabulinesep=1.2mm
\begin{tabu}{|l|c|c|c|} 
\hline
 & $de^3$ & $de^4$ & $de^5$ \\
\hline
$\rL_3\oplus \rA_2$ & $0$ & $0$ & $\m_{12}e^{12}$  \\ 
\hline
$\rL_4 \oplus \rA_1 $ & $0$ & $\m_{12}e^{12}$ & $e^1(\l_{12}e^2 + \l_{13}e^{3} + \m_{14}e^4)$ \\ 
\hline
$\rL_{5,1}$ & $0$ & $0$ & $\m_{12}e^{12} + \m_{34}e^{34}$ \\ 
\hline
$\rL_{5,2}$ & $0$ & $\m_{12}e^{12}$ & $  \m_{13}e^{13}$ \\ 
\hline
$\rL_{5,3}$ & $0$ & $\m_{12}e^{12}$ & $e^1(\l_{12}e^2+\l_{13}e^3+\mu_{14}e^{4})+\mu_{23}e^{23}$ \\ 
\hline
$\rL_{5,5}$ & $\m_{12}e^{12}$ & $e^1(\l_{12;4}e^2+\mu_{13}e^{3})$ & $\l_{12;5}e^{12}+\mu_{23}e^{23}$ \\ 
\hline
$\rL_{5,4}$ & $\m_{12}e^{12}$ & $ e^1(\l_{12;4}e^2+\mu_{13}e^{3})$ & $e^1(\l_{12;5}e^2+\l_{13}e^3+\mu_{14}e^{4})$ \\ 
\hline
$\rL_{5,6}$ & $\m_{12}e^{12}$ & $ e^1(\l_{12;4}e^2+\mu_{13}e^{3})$ & $e^1(\l_{12;5}e^2+\l_{13}e^3+\mu_{14}e^{4})+\mu_{23}e^{23}$ \\ 
\hline
\end{tabu}}
\end{center}

\end{lemma}

\begin{theorem}\label{5-harmonic}
If a 5-dimensional nilmanifold $\Gamma \backslash \rG$ admits left-invariant harmonic spinors, then $\frg=\rL_{5,j}$, $j=1,2,3,4,6$.
\end{theorem}

\begin{proof} 

Following the notation of Lemma \ref{5-classification}, we compute $\mu$ and $v$ defined as in Proposition \ref{4-form}. Obviously, $\mu$ is the sum of the squares of the parameters involved. In order to compute the vector $v$, we suppose that the nilpotent basis is positively oriented. This assumption does not depend on the existence of harmonic spinors. We summarize the result in the following table:
\begin{center}
\resizebox{14.75cm}{!}
{\tabulinesep=1.2mm
\begin{tabu}{|l|c|}
\hline
 & $v$ \\
\hline
$\rL_3\oplus\rA_2$ & $0$ \\ 
\hline
$\rL_4\oplus\rA_1 $ & $-2\mu_{12}\lambda_{13}e_1$ \\
\hline
$\rL_{5,1}$ & $2\mu_{12}\m_{34}e_5$ \\ 
\hline
$\rL_{5,2}$ & $-2\mu_{12}\m_{13}e_1$ \\ 
\hline
$\rL_{5,3}$ & $2(-\m_{12}\l_{13}e_1-\m_{12}\m_{23}e_2+\m_{14}\m_{23}e_5)$ \\
\hline
$\rL_{5,5}$ & $2(\m_{13}\l_{12;5}e_1 + \l_{12;4}\m_{23}e_2 -\mu_{12}\m_{13}e_3)$ \\ 
\hline
$\rL_{5,4}$ & $2(\m_{12}\m_{14}-\l_{12;4}\l_{13}+\m_{13}\l_{12;5})e_1$ \\ 
\hline
$\rL_{5,6}$ & $2((\m_{12}\m_{14}-\l_{12;4}\l_{13}+\m_{13}\l_{12;5})e_1-\m_{23}(\l_{12;4}e_2+\m_{13}e_3)+\m_{14}\m_{23}e_5)$ \\ 
\hline
\end{tabu}}
\end{center}

We now study, on each Lie algebra, the equation that determines the presence of left-invariant harmonic spinors: $\mu = \| v \|$.

$\rL_3\oplus \rA_2$ and $\rL_4 \oplus \rA_1 $ do not admit any left-invariant harmonic spinor because $\mu > \| v \|$. Left-invariant metrics admitting left-invariant harmonic spinors on $\rL_{5,1}$ are characterized by the equation $\mu_{12}=\pm \mu_{34}$. On the algebra $\rL_{5,2}$ are characterized by $\mu_{12}=\pm \mu_{13}$.

On the algebra $\rL_{5,3}$, the smallest eigenvalue of $16 \sD{}^2$ is
\[
 \l_{12}^2 + \mu_{12}^2 +  \lambda_{13}^2 + \mu_{14}^2 + \mu_{23}^2 - 2( \mu_{12}^2 (\lambda_{13}^2 + \mu_{23}^2) +  \mu_{14}^2\mu_{23}^2 )^{\frac{1}{2}} \geq 0.
\]
If the metric has harmonic spinors, necessarily $\l_{12}=0$. In addition, the previous condition leads us to $\l_{13}^2 = \m_{12}^2 - \m_{13}^2-\m_{14}^2 \pm  2 (\m_{14}^2\m_{23}^2 - \m_{12}^2 \m_{13}^2)^{\frac{1}{2}}$, whose solutions are $\l_{13}=0$, $\m_{23}^2>\m_{12}^2$ and $\m_{14}^2 = \m_{23}^2 - \m_{12}^2$.

On $\rL_{5,5}$ the smallest eigenvalue of $16\sD{}^2$ is,
\[ 
\mu_{12}^2 + \l_{12;4}^2 + \m_{13}^2 + \l_{12;5}^2 + \m_{23}^2 - 2(\m_{13}^2\m_{23}^2 + \l_{12;4}^2\m_{23}^2 + \l_{12;5}^2\m_{13}^2)^{\frac{1}{2}}\geq 0.
\]
Since this value is non-negative for every choice of the parameters, necessarily $\l_{12;4}^2 + \m_{13}^2 + \l_{12;5}^2 + \m_{23}^2 - 2(\m_{13}^2\m_{23}^2 + \l_{12;4}^2\m_{23}^2 + \l_{12;5}^2\m_{13}^2)^{\frac{1}{2}}\geq 0$. The smallest eigenvalue is therefore greater or equal to $\m_{12}^2> 0$. Consequently, the metric has no left-invariant harmonic spinors.

On $\rL_{5,4}$ the eigenvalues of $16 \sD{}^2$ are:
\[
(\mu_{12} \mp \mu_{14})^2 + (\lambda_{12;4} \pm \lambda_{13})^2 + (\mu_{13} \mp \lambda_{12;5})^2.
\]
Metrics which admit left-invariant harmonic spinors are such that: $\mu_{12}= \pm \mu_{14}$, $\lambda_{12;4}= \mp\lambda_{13}$ and $\mu_{13} = \pm \lambda_{12;5}$.

Finally, a metric on $\rL_{5,6}$ 
has left-invariant harmonic spinors if and only if:
\begin{align*}
 ( \mu_{12}^2 +& \lambda_{12;4}^2 + \mu_{13}^2 + \lambda_{12;5}^2 + \lambda_{13}^2 + \mu_{14}^2 + \mu_{23}^2)^2 = \\
&=  4\left( \mu_{14}^2\mu_{23}^2 + (-\mu_{13}\lambda_{12;5} + \lambda_{13}\lambda_{12;4} - \mu_{12}\mu_{14})^2 + \lambda_{12;4}^2\mu_{23}^2 + \mu_{13}^2\mu_{23}^2 \right).
\end{align*}

We now show that this equation has solutions. If we suppose that $\lambda_{12;4}=0$ then the condition $\l_{13}=0$ is necessary for the presence of harmonic spinors. Moreover, the previous equation leads us to: $\m_{23}^2= \m_{13}^2 + \m_{14}^2 - \m_{12}^2 - \l_{12;5}^2 \pm 2 \imag (\l_{12;5}^2\m_{14}^2- \m_{12}^2\m_{13}^2) $.
Therefore the solutions are $\l_{12;5}^2= \dfrac{\m_{13}^2\m_{12}^2}{\m_{14}^2}$ and $\m_{23}^2 = \frac{1}{\m_{14^2}}(\m_{14}^2-\m_{12}^2)(\m_{13}^2 + \m_{14}^2)$ with $\m_{14}^2>\m_{12}^2$.
\end{proof}

Lemma \ref{5-classification} is a list in which one fixes an orthonormal basis of $\RR^5$ and varies the Lie bracket within an isomorphism class of Lie brackets. 

From Lemma \ref{alpha-v} and the proof of Theorem \ref{5-harmonic} we obtain that for every harmonic invariant $\SU(2)$ structure on nilmanifolds with Lie algebras $\rL_{5,1}$, $\rL_{5,2}$ and $\rL_{5,4}$, the direction of the $1$-form $\a$ does not depend on the isomorphism class of the Lie bracket. We analyze each case separately, giving an example of the forms that determine the structure which have been computed using the representation fixed at the beginning of the section. We also suppose that the basis $(e_1,e_2,e_3,e_4,e_5)$ is positively oriented.

On the algebra $\rL_{5,1}$, $\a$ is parallel to $e^5$, in particular, if $\mu_{12}=\pm \mu_{34}$ then $\a= \mp e^5$. Then $\a$ is contact because $d\a=\m_{34}(\pm e^{12} + e^{34})$. Moreover, $\xi= \la e_1,\dots,e_4\ra$ and therefore, $d\o_k=0$ for $k=1,2,3$.

If $\mu_{12}=-\mu_{34}$, then $\ker(\jota + \a \cdot)= \ker(\jota + e_5\cdot )= \la \p_1,\p_2,\p_3, \p_4 \ra$. If we take $\eta=\p_1$, then $\o_1=e^{12}+ e^{34}$, $\o_2= e^{14}+ e^{23}$ and $\o_3= e^{13}- e^{24}$. 
Thus, $d\a=\tau_2^4 \in \frsu(2)$ with $\tau_2^4= \m_{12}(e^{12}-e^{34})$.
Since $d\o_1=0$ and $d(\a\wedge \o_2)=d(\a\wedge \o_3)=0$, the structure is hypo.
In the same manner, when $\m_{12}=\m_{34}$ we take $\eta=\p_5$ and obtain $\o_1= -e^{12}+ e^{34}$, $\o_2= e^{14}- e^{23}$ and $\o_3=-e^{13}+ e^{24}$. 
Again, $d\a=\tau_2^4 \in \frsu(2)$ with $\tau_2^4= \m_{12}(e^{12}+ e^{34})$.

On the algebras $\rL_{5,2}$ and $\rL_{5,4}$, $\a$ is parallel to $e^1$ and, consequently, $d\a=0$. These algebras are quasi-abelian, that is, they have a codimension-$1$ abelian ideal, which is $\xi=\la e_2,e_3,e_4,e_5\ra$. In particular, taking into account the equations in terms of forms of harmonic structures, $d\o_k= \a \wedge \tau_2^k$. Thus, $d(\o_k\wedge \a)=0$. 
In the case $\a=-e^1$ we choose $\eta=2^{-\frac{1}{2}}(\p_1 + \p_5) \in \ker(\jota - e_1)$. 
Therefore, $\o_1=-e^{25} + e^{34}$, $\o_2=e^{23}-e^{45}$ and $\o_3=e^{24} + e^{35}$.
On the one hand, the nilpotency of the basis implies that $i(e_5)d\o_1=0$. On the other, $i(-e_1)d\o_1=\tau_2^1$ which is $0$ or non-degenerate on $\xi$. Hence, $d\o_1=0$. The same argument holds for $d\o_3$ on $\cN_{5,5}$ because $e^3$ is closed. Thus, the structure is of type hypo and the torsions which may be non-zero are $\tau_2^2$ and $\tau_2^3$; we compute them:

\begin{enumerate}
\item On $\rL_{5,2}$ the condition $\a=-e^1$ implies $\mu_{12}=-\mu_{13}$. Then, $d\o_2=\m_{13}(e^{125}+e^{134})$ so that the unique non-zero torsion is $\tau_2^2=\m_{13}(e^{25}+e^{34})$.
\item On $\rL_{5,4}$ the condition $\a=e^1$ implies $\m_{12}=\m_{14}$, $\l_{13}=-\l_{12;4}$ and $\m_{13}=\l_{12;5}$. Then, $d\o_1=0$, $d\o_2=e^1(\l_{13}(e^{25}+ e^{34}) + \m_{13}(e^{24}-e^{35}))$ and $d\o_3=\m_{12}(e^{25}+e^{34})$.
\end{enumerate}

On $\rL_{5,3}$ metrics with harmonic spinors verify $\l_{12}=\l_{13}=0$ and $\m_{14}^2 = \m_{23}^2 - \m_{12}^2>0$. Therefore, $v=2(-\m_{12}\m_{23}e_2 \pm \m_{23}(\m_{23}^2 - \m_{12}^2)^{\frac{1}{2}}e_5)$. Thus, $d\a$ is proportional to $\m_{14} e^{14}+ \mu_{23}e^{23}$ and harmonic invariant structures are contact.

\begin{remark}
 Fern\'andez' first example of a balanced $\Spin(7)$-manifold was a nilmanifold $\Gamma\backslash \rG$ with $\frg=\rL_{5,2}\oplus \rA_3$.
\end{remark}

\subsection{6-dimensional nilmanifolds}\label{subsec:6-nilm}

We fix the irreducible representation of $\Cl_6$ described in Section \ref{Cl-representations} and denote by $\jota$ 
the Clifford multiplication by the volume form, which anticommutes with the Clifford product with a vector. As in the $5$-dimensional case we have the following:

\begin{proposition} \label{6-square}
Let $(e_1,\dots,e_6)$ be an orthonormal nilpotent frame of $\frg$ and let $\p$ be an invariant spinor. Then
$16 \sD{}^2\p= \m \p +  \gamma j\p$, where $\m = \sum {\|de^i\|^2}$ and 

\begin{align*}
 \gamma =& \sum_{l=5}^6\star \left((de^l\wedge de^l) + \sum_{i=3}^4{de^i \wedge i(e_i)de^l \wedge e^l} \right) \\
    &-  \sum_{l=5}^6\star\left(\sum_{i=3}^4{\sum_{k=1}^3 {i(e_k)de^i \wedge i(e_k)(de^l|_{\la e_i \ra^\perp})\wedge e^{il}}} \right) \\
    &+\star( de^5 \wedge i(e_5)de^6 \wedge e^6) -\star\left(\sum_{k=1}^4 {i(e_k)de^5 \wedge i(e_k)(de^6|_{\la e_5 \ra^\perp} )}  \right).
\end{align*}
In addition, the restriction of the operator $\sD{}^2$ over the space of invariant spinors has eight eigenspaces, 
$\D_j$, associated to $\pm \l_1, \pm \l_2, \pm \l_3, \pm \l_4$ for some $0\leq \l_1 \leq \l_2 \leq \l_3 \leq \l_4$ 
and $\jota$ restricts to a map, $\jota \colon \D_{\l_j} \to \D_{-\l_j}$.
\end{proposition}


\subsubsection{Decomposable algebras} 
Except for $\rL_3\oplus \rL_3=(0,0,0,0,12,34)$, the structure constants of decomposable Lie algebras can easily be obtained by those in dimension 5, listed above. We proceed to obtain a metric classification of such Lie algebras, characterizing the structure equations in terms of an orthonormal basis. 

\begin{lemma} \label{6-decomposable-classification}
The list of $6$-dimensional decomposable metric nilpotent algebras is: \\
\noindent \resizebox{15cm}{!} 
{\tabulinesep=1.2mm
\begin{tabu}{|l|c|c|c|} 
\hline
                                        &    $de^4$     &            $de^5$ & $de^6$ \\
\hline
$\rL_3\oplus \rA_3$  &   $0$      &     $0$         & $\m_{12}e^{12}$  \\ \hline

$\rL_3\oplus \rL_3$   &     $0$         & $\m_{12}e^{12} + \lambda_{13;5}e^{13}$ & $e^3(\mu_{14}e^4 + \lambda_{13;6}e^{1} + \lambda_{23}e^2)$ \\ \hline

$\rL_4 \oplus \rA_2 $            &    $0$      & $\m_{12}e^{12}$ & $e^1(\l_{12}e^2 + \l_{13}e^{3} + \m_{15}e^5)$ \\ \hline

$\rL_{5,1}\oplus \rA_1$      & $0$        & $0$ &  $\m_{12}e^{12} + \m_{34}e^{34}$ \\ \hline

$\rL_{5,2}\oplus \rA_1$      &    $0$        & $\m_{12}e^{12}$ & $  \m_{13}e^{13}$ \\ \hline

$\rL_{5,3}\oplus \rA_1$        &    $0$        &   $\m_{12}e^{12}$ & $e^1(\l_{12}e^2 + \l_{13}e^3 + \l_{14}e^4 + \m_{15}e^{5})+ \m_{23}e^{23}$ \\ \hline

$\rL_{5,5}\oplus \rA_1$       &         $\mu_{12}e^{12}$   &   $\lambda_{12;5} e^{12} + \l e^{23} + \mu_{14}e^{14}$ & $\lambda_{12;6} e^{12} + \l e^{13} + \mu_{24}e^{24}$ \\ \hline

$\rL_{5,4}\oplus \rA_1$        &       $\m_{12}e^{12}$ & $ e^1(\l_{12;5}e^2 + \mu_{14}e^{4} + \l_{13}e^3)$ & $e^1(\l_{12;6}e^2 + \l_{13}e^3 + \l_{14}e^4 + \mu_{15}e^{5})$ \\ \hline

$\rL_{5,6}\oplus \rA_1$       &   $\m_{12}e^{12}$ & $ e^1(\l_{12;5}e^2 + \m_{14}(\l_{13;4}e^{3} + e^4))$ & $e^1(\l_{12;6}e^2 + \l_{13;6}e^3  + \l_{14}e^{14}+ \mu_{15}e^{5}) + \mu_{24} e^2 (\l_{13}e^{3} + e^4) $ \\ \hline

\end{tabu} }

\end{lemma}

\begin{proof}

The equations for $\rL_3\oplus \rL_3$ are obtained from a basis $(x^1,\dots,x^6)$ associated to the stucture equations $(0,0,0,0,12,34)$. First observe that we can suppose that $x^i$ is orthogonal to $x^{i+1}$ for $i\in \{1,3\}$ and that $x^1$ is orthogonal to $x^3$.
The Gram-Schmidt process allows us to obtain an orthonormal basis $e^1=\frac{x^1}{\|x^1\|}$, $e^3=\frac{x^3}{\|x^3\|}$, $e^2 = \mu_{22} x^2 + \mu_{23} e^3$ and $e^4= \mu_{44}x^4 + \l_{14}e^1 + \l_{24}e^2 + \l_{34}e^3$.

Finally take two orthogonal and unit-length forms $e^5, e^6 \in \ker(d)^\perp$ with $de^{5}=x^{12}$. 

The rest of the algebras can be decomposed as $\rL_5\oplus \rA_1$, where $\rL_5$ is a $5$-dimensional nilpotent Lie algebra. Let $d_5$ be the corresponding differential. Let $dt$ be a generator of $\rA_1^*$ and observe that $\ker(d)=\ker(d_5)\oplus \la dt \ra$ and
$d \colon d^{-1}(\L^2 \ker(d))\to \L^2 \ker(d_5)$. Therefore, a unit-length $1$-form $\alpha \in \ker(d)$ orthogonal to $\ker(d_5)$ verifies $i(\a^\sharp)d\b=0$ for all $\beta \in d^{-1}(\L^2\ker d)$. 

If the Lie algebra is $2$-step there the decomposition $\rL_5\oplus \rA_1$ is orthogonal and the equations follow from Lemma \ref{5-classification}.

The equations for $\rL_{5,6}\oplus \rA_1$, $\rL_{5,4}\oplus \rA_1$ and $\rL_{5,3}\oplus \rA_1$ can be arranged using the Gram-Schmidt process, starting with an orthonormal basis $(e^1,\dots,e^k, \a)$ with $e^i \in \ker(d_5)$.

To obtain the equations for $\rL_{5,5}\oplus A_1$ consider $F_1=d^{-1}(\L^2\ker d)\cap \ker(d)^{\perp}$ 
and $F_2=d^{-1}(\L^2 F_1) 
\cap F_1^\perp$. Let $\pi$ the plane generated in $(\rL_{5,5}\oplus A_1)^*$ by $dF_1$ and observe that there is an isomorphism $\tilde{d} \colon F_2 \longmapsto \pi \otimes F_1$ obtainted from $d$ and the projection of the space of closed forms to $\pi \otimes F_1$. 
Take $e^4 \in F_1$ unit-length and let $e^5,e^6 \in F_2$ and $e^1,e^2\in \pi$ orthonormal such that $\tilde{d}e^5= \mu_{14} e^{14}$ and $\tilde{d}e^6= \mu_{24} e^{24}$. Define the map $\pi \to \pi$, $ \b \longmapsto \star p(d \tilde{d}^{-1}(\b \otimes e^4))$, where $\star$ is the Hodge star and $p\colon \L^2 \ker(d) \oplus (\pi \otimes F_1 ) \to \L^2 \ker(d) \cap dF_1^\perp$
is the orthogonal projection.
This map is diagonal with eigenvalue $\l$ (see \cite[pp.\ 1017-1018]{BM2}),
so that $de^{5}=\lambda_{12;5} e^{12} + \l e^{23} + \mu_{14}e^{14}$ and 
$de^{6}=\lambda_{12;6} e^{12} + \l e^{13} + \mu_{24}e^{24}$.
\end{proof}

We describe the set of metrics on $\rL_3\oplus \rL_3$ with harmonic spinors.

\begin{lemma} \label{L3+L3}
Following the notation of Lemma \ref{6-decomposable-classification}, metrics with 
harmonic spinors on $\rL_3\oplus \rL_3$ are those which verify one of the following conditions:
\begin{enumerate}
\item $\l_{23}=0$, $\l_{13;6}= \sigma_1 \m_{12}$ and $\l_{13;5}=\sigma_2 \m_{34}$, for some $\sigma_1,\s_2\in \{\pm 1\}$.
\item $
4\l_{23}^2(\l_{13}^2 + \m_{12}^2)= \m_{12}^2 + \l_{13;5}^2 + \l_{13;6}^2 + \l_{23}^2 + \m_{34}^2 - 4( \sigma \m_{12}\l_{13;6} + \l_{13;5}\m_{34})^2 
$
for some $\sigma \in \{ \pm 1 \}$.
\end{enumerate}
\end{lemma}
\begin{proof}
We first take an orthonormal basis $(e^1,\dots,e^6)$ associated to the structure equations given in Lemma \ref{6-decomposable-classification}. Then, $\mu$ is the sum of the squares of the parameters involved and supposing that the basis is positively oriented,
$$
\gamma=-2(\m_{12}\l_{13;6}e^{14} + \l_{13;5}\l_{23}e^{34} + \m_{12}\l_{23}e^{24} - \l_{13;5}\m_{34}e^{23}).
$$

Note that the operators $e^{14}\jota$ and $e^{23}\jota$ commute. 
Define the operator 
 $$
 A= -2(\l_{13;5}\l_{23}e^{34} + \m_{12}\l_{23} e^{24})\jota\cdot
 $$ 
and observe that it anticommutes with the previous operators and that 
$A^2= 4\l_{23}^2(\l_{13;5}^2 + \m_{12}^2)\rI$. We distinguish two cases:

\begin{itemize}
\item If $\l_{23}=0$ then $A=0$ and the eigenvalues of $\sD{}^2$ are
$(\m_{12}^2 \pm \l_{13;6})^2 + (\l_{13;5} \pm \m_{24})^2$.
Therefore, the metric has harmonic spinors if $\l_{13;6}= \pm \m_{12} \neq 0$ and $\l_{13;5}=\pm \m_{34} \neq 0$.

\item If $\l_{23}\neq 0$ then $A$ is invertible. Denote $\mu= \m_{12}^2 + \l_{13;5}^2 + \l_{13;6}^2 + \l_{23}^2 + \m_{34}^2$. Let $\D_{\pm}$ be the eigenspaces associated to the eigenvalue $\pm 1$ of $e^{14}\jota$ 
and decompose $\D_{\pm}=\D_{\pm}^{+}\oplus \D_{\pm}^{-}$ according to the eigenspaces of $e^{23}\jota$.
Note that $A(\D_{\pm}^{+})=\Delta_{\mp}^{-}$ and that $A^2= 4\l_{23}^2(\l_{13}^2 + \m_{12}^2)\rI$. Thus, the eigenvalues are of the form $\p_{\pm}^{+} + \p_{\mp}^{-}$ with $\p_{\pm}^{+} \in \D_{\pm}^{+}$ and $\p_{\mp}^{-} \in \D_{\mp}^{-}$.
The eigenvalue $0$ occurs on $\D_{+}^{+}\oplus \D_{-}^{-}$ if and only if:
\begin{align*}
 A\p_{+}^{+} = & (\m - 2\m_{12}\l_{13;6} + 2\l_{13;5}\m_{34})\p_{-}^{-}\, , \\
 A\p_{-}^{-} = & (\m + 2\m_{12}\l_{13;6} - 2\l_{13;5}\m_{34})\p_{+}^{+}\, .
\end{align*}
This implies that $4\l_{23}^2(\l_{13}^2 + \m_{12}^2) = (\m - 2\m_{12}\l_{13;6} + 2\l_{13;5}\m_{34})(\m + 2\m_{12}\l_{13;6} - 2\l_{13;5}\m_{34})$.
Moreover, if this equation holds we can take $\p_{+}^{+} \in \D_{\pm}^{+}$, define $\p_{-}^{-}=(\m + 2\m_{12}\l_{13;6} - 2\l_{13;5}\m_{34})A^{-1}\p_{+}^{+}$. Then,
 \begin{align*}
 A\p_{+}^{+} &= (\m + 2\m_{12}\l_{13;6} - 2\l_{13;5}\m_{34})^{-1} A^2\p_{-}^{-} \\ &
 = (\m - 2\m_{12}\l_{13;6} + 2\l_{13;5}\m_{34})\p_{-}^{-}\, .
 \end{align*}
We can do a similar analysis on $\D_{+}^{-}\oplus \D_{-}^{+}$ to conclude that the metric has harmonic spinors if and only if 
 \begin{align*}
 4\l_{23}^2 & (\l_{13}^2 + \m_{12}^2)\\ &
= \m_{12}^2 + \l_{13;5}^2 + \l_{13;6}^2 + \l_{23}^2 + \m_{34}^2 - 4( \sigma \m_{12}\l_{13;6} + \l_{13;5}\m_{34})^2 \, ,
 \end{align*}
for some $\sigma \in \{ \pm 1 \}$.
If $\m_{12}=1$, this equation has solutions if and only if,
$ 1 + \l_{13;5}^2 + \l_{13;6}^2 + \m_{34}^2 - 4(\l_{13;6} + \l_{13;5}\m_{34})^2 >0$.
This inequality holds taking the parameters small enough.
\end{itemize}
\end{proof}

The other decomposable cases can be obtained by taking into account the results of the previous sections.
It is clear from Theorem \ref{5-harmonic} and Lemma \ref{6-decomposable-classification} that the algebras $\rL_3 \oplus \rA_3$ and $\rL_4\oplus \rA_2$ do not admit left-invariant harmonic spinors and that $\rL_{5,j}\oplus \rA_1$ has harmonic spinors for $j\neq 5$. Finally take an orthonormal basis $(e^1,\dots,e^6)$ associated to the structure equations of $\rL_{5,5}\oplus \rA_1$ given in Lemma \ref{6-decomposable-classification} and suppose $\m_{12}=1$. Now we write the Dirac operator using the formula obtained in Corollary \ref{nilpotent-Dirac-operator} 
and then we use the fix representation to obtain an endomorphism of the spinoral bundle. The metric has left-invariant harmonic spinors if and only if the determinant of the endomorphism is $0$. Solving the equation we get:
$$
 \l = \frac{1}{2}(1 +(\m_{14} + \m_{24})^2 )^{-\frac{1}{2}}((1+ \l_{12;5}^2 + \l_{12;6}^2 + \m_{14}^2- \m_{24}^2)^2 - 4\l_{12;6}^2 + 4\m_{24}^2 )^{\frac{1}{2}}\, .
 $$
But the number on the square root is obviously positive if $\l_{12;6}=0$. Therefore, there are metrics with harmonic spinors.
 
Hence we have proved:
 \begin{theorem}\label{6-harmonic-decomposable}
 Let $\Gamma \backslash \rG$ be a non-abelian $6$-dimensional nilmanifold with $\frg$ decomposable. Then, unless $\frg$ equals $\rL_3 \oplus \rA_3$ or $\rL_4\oplus \rA_2$, $\Gamma \backslash \rG$ admits an invariant metric with left-invariant harmonic spinors. 
 \end{theorem}

\subsubsection{Non-decomposable algebras}\label{6-harmonic-non-decomposable}
Using the fixed representation of $\Cl_6$ we are able to find a metric with harmonic spinors on each nilmanifold associated to a non-decomposable Lie algebra. We follow the same procedure that we used to determine metrics with left-invariant harmonic spinors on $\rL_{5,5}\oplus \rA_1$. In many cases we will not be able to determine the roots of the polynomial in terms of the parameters. Thus, we will have to make some choices as the following example explains:

We consider the algebra $\rL_{6,7}$, which has structure equations $(0,0,0,12,13,15+24)$. We first declare the canonical basis orthonormal and compute the Dirac operator. One can show that this metric does not have left-invariant harmonic spinors. Neither does any metric constructed by declaring orthonormal a basis which is obtained by rescaling the canonical basis.

Now we proceed to write the structure equations by means of an orthonormal basis with respect to a metric. 
First, write $F_1=\ker(d)$, $F_2=d^{-1}(\L^2 F_1)$ and $F_3= d^{-1}(\L^2 F_2)=\rL_{6,7}$. 
One can take an orthonormal basis of $F_2$ such that $de^4=\m_{13}e^{12}$ and $de^{5}=\m_{13}e^{13}$. Now take $e^6$ orthogonal to $F_2$, then according to \cite{BM2}, $de^6$ is a closed form of $\L^2F_2$ such that $e^1\wedge(de^6)^2=0$, $e^1\wedge de^6 \not\in \L^3 F_1$ and $de^6 \notin \ker(d)\otimes F_2$. Those equations imply: 
 \begin{align*}
 de^6=&\,  \l_{12}e^{12} + \l_{13}e^{13} + \l_{14}e^{14} + \l_{15}e^{15} \\ &+ \l_{23}e^{23} +\l_{24}e^{24} 
 + \l_{35}e^{35} + \left(\dfrac{\l_{24}\l_{35}}{\m_{12}\m_{13}}\right)^{\frac{1}{2}}(\m_{13}e^{34} + \m_{12}e^{25}), 
 \end{align*}
with $\l_{24}\l_{35}\geq 0$ and $-\l_{14}\left(\dfrac{\l_{24}\l_{35}}{\m_{12}\m_{13}}\right)^{\frac{1}{2}} \m_{12} + \l_{15}\l_{24}\neq 0$.
We choose $\l_{35}=0$ and therefore, $de^6= \l_{12}e^{12} + \l_{13}e^{13} + \l_{14}e^{14} + \l_{15}e^{15} + \l_{24}e^{24} $ with $\l_{15}\l_{24}\neq 0$. We fix $1=\m_{13}=\m_{12}=\l_{15}=\l_{24}$ and vary the rest of the parameters. 

The choice $\l_{12}=1=\l_{23}$ leads to the condition that $\l_{13}$ is a root of the polynomial
$ Z^8 + 8(\l_{14}^2 +8)Z^6 + 16\l_{14}^2 + 24 \l_{14}^2 + 32)Z^4 + 32\l_{14}^3Z^3 + 4(\l_{14}^6+24\l_{14}^4 + 128\l_{14}^2)Z^2 + (16\l_{14}^5 + 128 \l_{14}^3)Z + \l_{14}^8+8\l_{14}^6+32\l_{14}^4$. 
Hence, $(\l_{13},\l_{14})=(0,0)$ is a solution.

We finish with a list of the non-decomposable metric nilpotent Lie algebras in dimension $6$ which admit a harmonic spinor.


\noindent \resizebox{15cm}{!}
{\tabulinesep=0.7mm
\begin{tabu}{|l|l|c|c|c|c|}
\hline
&  & $de^3$  &   $de^4$     &    $de^5$ & $de^6$  \\ 
\hline
$\rL_{6,1}$ & $(0,0,0,0,12,13 + 24)$ & $0$ & $0$ & $e^{12}$ & $2e^{13} + e^{24}$ \\ \hline

$\rL_{6,2}$ & $(0,0,0,0,13-24,14 + 23)$& $0$ & $0$ & $e^{13}-e^{24}$& $e^{14} + e^{23}$ \\ \hline 

$\rL_{6,3}$ & $(0,0,0,0,12,15+34)$& $0$ & $0$ & $e^{12}$& $ e^{14} + e^{15} + e^{34}$ \\ \hline 

$\rL_{6,4}$ & $(0,0,0,12,13,23)$& $0$ &  $e^{12}$ & $e^{13}$ & $ 2 e^{23}$ \\ \hline 

$\rL_{6,5}$ & $(0,0,0,12,13,14)$& $0$& $e^{12}$ & $2^{\frac{1}{2}}e^{13}$ & $e^{14}$ \\ \hline 

$\rL_{6,6}$ & $(0,0,0,12,13,24)$& $0$ & $e^{12}$ & $e^{13}$ & $2e^{13} + 3^{\frac{1}{2}}e^{24} + e^{23}$ \\ \hline 

$\rL_{6,7}$ & $(0,0,0,12,13,15 + 24)$& $0$ & $e^{12}$ & $e^{13}$ & $e^{12}+ e^{15} +e^{23} +e^{24} + e^{23}$
 \\ \hline 

$\rL_{6,8}^+$ & $(0,0,0,12,13, 24 + 35)$&  $0$ & $e^{12}$ & $e^{13}$ & $e^{24} + e^{35}$
 \\ \hline 

$\rL_{6,8}^-$ & $(0,0,0,12,13, 24-35)$& $0$ & $e^{12}$ & $e^{13}$ & $- 2e^{23} +e^{24} - e^{35}$
 \\ \hline

$\rL_{6,9}$ & $(0,0,0,12,13,14+ 23)$&  $0$ & $e^{12}$ & $e^{13}$ & $e^{14} + e^{23} + (2(2^\frac{1}{2}-1))^\frac{1}{2} e^{12}$
 \\ \hline 
 
$\rL_{6,10}$ & $(0,0,0,12,14,23+24)$&  $0$ & $e^{12}$ & $e^{14}$ & $e^{23} + e^{24}$
 \\ \hline 
 
$\rL_{6,11}$ & $(0,0,0,12,14,13+24)$&  $0$ & $e^{12}$ & $e^{14}$ & $e^{13} + e^{24}$
 \\ \hline 
 
$\rL_{6,12}$ & $(0,0,0,12,14 + 23,13- 24)$&  $0$ & $2^{-\frac{1}{2}}e^{12}$ & $2^{\frac{1}{2}}e^{14}+ e^{23} $ & $e^{13} - 2^{\frac{1}{2}} e^{24}$
 \\ \hline 

$\rL_{6,13}$ & $(0,0,0,12,14,15 + 23)$& $0$ & $e^{12}$ & $e^{14}$ & $e^{15}+ 2^{\frac{1}{2}}e^{13} + e^{23} $ 
 \\ \hline 
 
$\rL_{6,14}$ & $(0,0,0,12,14,15 + 23 + 24)$& $0$ & $e^{12}$ & $e^{14}- \frac{7}{4}e^{13}$ & $e^{15}+ e^{24} - \frac{3}{4} e^{23} + 2 e^{12} $ 
 \\ \hline 
 
$\rL_{6,15}$ & $(0,0,0,12,14 + 23, 15- 34)$& $0$ & $e^{12}$ & $e^{14}+e^{23}$ & $\frac{1}{4}(e^{15}+ e^{34}) $ 
 \\ \hline

$\rL_{6,16}$ & $(0,0,12,13,23,14)$&  $e^{12}$ & $e^{13}$ & $e^{23}$ & $e^{14}$ 
 \\ \hline

$\rL_{6,17}^+$ & $(0,0,12,13,23,14+25)$&  $e^{12}$ & $e^{13}$ & $e^{23}$ & $e^{14} + e^{24} + e^{12} + 2^{\frac{1}{2}}e^{23}$ 
 \\ \hline 
 
$\rL_{6,17}^-$ & $(0,0,12,13,23,14-25)$&  $e^{12}$ & $e^{13}$ & $e^{23}$ & $e^{14} - e^{25} - e^{12} + 2^{\frac{1}{2}}e^{23}$ 
 \\ \hline 
 
$\rL_{6,18}$ & $(0,0,12,13,14,15)$&  $e^{12}$ & $e^{13}$ & $\frac{1}{5}(e^{14}+ e^{12})$ & $\frac{1}{5}(e^{12} + e^{14} + 46^{\frac{1}{2}}e^{15})$ 
 \\ \hline

$\rL_{6,19}$ & $(0,0,12,13,14,15 +23)$&  $e^{12}$ & $e^{13}$ & $e^{14}$ & $e^{15} + e^{23} + e^{12}$ 
 \\ \hline 
  
$\rL_{6,20}$ & $(0,0,12,13,14,15 -34)$&  $e^{12}$ & $e^{13}$ & $e^{14}$ & $e^{25} - e^{34} + 5^{\frac{1}{2}}e^{12}$ 
 \\ \hline 

$\rL_{6,21}$ & $(0,0,12,13,14+23,15+24)$&  $e^{12}$ & $e^{13}$ & $\frac{1}{m}(e^{14}+ e^{23})$ & $me^{15} + e^{24}$ 
 \\ \hline 
 
$\rL_{6,22}$ & $(0,0,12,13,14+23,15-34)$&  $e^{12}$ & $e^{13}$ & $e^{14}+ e^{23}$ & $e^{25} - e^{34}+ (1 + 5^{\frac{1}{2}})e^{12}$ 
 \\ \hline
 
\end{tabu} }

\noindent where
$m=\frac{\sqrt{3} \left( (459+12\sqrt{177})^{\frac{1}{3}} ((459+12*\sqrt{177})^{\frac{2}{3}}+6(459+12\sqrt{177})^{\frac{1}{3}}+57)\right){}^{\frac{1}{2}} }{3(459+12\sqrt{177})^{\frac{1}{3}}}$.

\subsection{8-dimensional nilmanifolds with balanced \texorpdfstring{$\Spin(7)$}{Lg} structures}\label{subsec:8-balanced}

The results collected so far allow us to obtain examples of invariant balanced $\Spin(7)$-structures on nilmanifolds $N_k\times T^{8-k}$ with $N_k$ a $k$-dimensional nilmanifold, $k=5,6$. By considering $N_k\times T^{7-k}$, one obtains a 7-dimensional nilmanifold with a spin-harmonic $\rG_2$-structure. If $M$ is any 7-dimensional manifold endowed with a spin-harmonic $\rG_2$-structure, then $M\times S^1$ admits a balanced $\Spin(7)$-structure. According to Theorem \ref{G2-Dirac}, every closed $\rG_2$-structure is spin-harmonic and a coclosed $\rG_2$-structure is spin-harmonic if and only if it is of pure type $\chi_3$. Now 7-dimensional nilpotent Lie algebras with closed and coclosed $\rG_2$-structures are classified by Conti-Fern\'andez \cite{Conti-Fernandez} and Bagaglini \cite{Bagaglini} respectively. We show that not all our examples of balanced $\Spin(7)$ nilmanifolds can be obtained by Conti-Fern\'andez and Bagaglini. To do this we compare decomposable 7-dimensional Lie algebras admitting closed, coclosed and spin-harmonic $\rG_2$-structures in the table below. 

We have seen in Theorem \ref{6-harmonic-decomposable} that $\rL_3\oplus\rA_3$ and $\rL_4\oplus\rA_2$ do not admit any metric with harmonic spinors; we show that the same happens when we add abelian factors of dimension 1 and 2 to these Lie algebras.

\begin{proposition} The Lie algebras $L_3\oplus A_4$, $L_3\oplus A_5$, $L_4\oplus A_3$ and $L_4\oplus A_4$ do not admit any metric with harmonic spinors.
\end{proposition}

\begin{proof} We prove the result for $L_3\oplus A_4$ and $L_4\oplus A_3$, the other cases being similar. Let us write the structure equations in term of a suitable orthonormal basis $(e^1,\dots, e^7)$ of each Lie algebra.
\begin{enumerate}

\item[1.] For $L_3\oplus A_4$ the structure equations are $de^i=0$ for $i=1,\ldots,6$, and $de^7=\mu e^{12}$, for some $\mu\neq 0$. One computes that $\sD\phi=\mu e^{127}\phi$, which has no kernel.
\item[2.] The structure equations of $L_4\oplus A_3$ are $de^i=0$ for $i=1,\ldots,5$
\[
 de^6=\mu_{12}e^{12} \quad \textrm{and} \quad  de^7=\mu_{16}e^{16}+\lambda_{12}e^{12}+\lambda_{13}e^{13}\,.
\]
With this one computes $\sD\phi=e^{1}(\mu_{12}e^{26}+\mu_{16}e^{67}+\lambda_{12}e^{27}+\lambda_{13}e^{37})\phi$. Note that $\mu_{16}e^{167}\phi$ is orthogonal to $e^{1}(\mu_{12}e^{26}+\lambda_{12}e^{27}+\lambda_{13}e^{37})\phi$ hence, since $\mu_{16}\neq 0$, the kernel of the Dirac operator is trivial.

\end{enumerate}
\end{proof}

\vfill \pagebreak

\begin{center}
{\tabulinesep=1.2mm
\begin{tabu}{|l|c|c|c|} 
\hline
  &    closed & coclosed & spin-harmonic \\
\hline
$\rL_3\oplus\rA_4$ & $\times$ & $\checkmark$ & $\times$\\ 
\hline
$\rL_3\oplus\rL_3\oplus A_1$ & $\checkmark$ & $\checkmark$ & $\checkmark$\\ \hline
$\rL_4\oplus\rA_3$ & $\times$ & $\times$ & $\times$ \\ 
\hline
$\rL_{5,1}\oplus\rA_2$ & $\times$ & $\checkmark$ & $\checkmark$ \\ \hline
$\rL_{5,2}\oplus\rA_2$ & $\checkmark$ & $\checkmark$ & $\checkmark$ \\ \hline
$\rL_{5,3}\oplus\rA_2$ & $\times$ & $\checkmark$ & $\checkmark$ \\
\hline
$\rL_{5,5}\oplus\rA_2$ & $\times$ & $\times$ & $\checkmark$ \\ \hline
$\rL_{5,4}\oplus\rA_2$ & $\times$ & $\checkmark$ & $\checkmark$ \\ \hline
$\rL_{5,6}\oplus\rA_2$ & $\times$ & $\checkmark$ & $\checkmark$ \\ \hline
$\rL_{6,1}\oplus\rA_1$ & $\times$ & $\checkmark$ & $\checkmark$ \\ 
\hline
$\rL_{6,2}\oplus\rA_1$ & $\times$ & $\checkmark$ & $\checkmark$ \\
\hline
$\rL_{6,3}\oplus\rA_1$ & $\times$ & $\checkmark$ & $\checkmark$ \\
\hline
$\rL_{6,4}\oplus\rA_1$ & $\times$ & $\checkmark$ & $\checkmark$ \\
\hline 
$\rL_{6,5}\oplus\rA_1$ & $\times$ & $\times$ & $\checkmark$ \\
\hline
$\rL_{6,6}\oplus\rA_1$ & $\times$ & $\checkmark$ & $\checkmark$ \\
\hline 
$\rL_{6,7}\oplus\rA_1$ & $\times$ & $\times$ & $\checkmark$ \\
\hline
$\rL_{6,8}^+\oplus\rA_1$ & $\times$ & $\checkmark$ & $\checkmark$ \\
\hline
$\rL_{6,8}^-\oplus\rA_1$ & $\times$ & $\checkmark$ & $\checkmark$ \\
\hline
$\rL_{6,9}\oplus\rA_1$ & $\times$ & $\checkmark$ & $\checkmark$ \\ 
\hline 
$\rL_{6,10}\oplus\rA_1$ & $\times$ & $\times$ & $\checkmark$ \\ 
\hline 
$\rL_{6,11}\oplus\rA_1$ & $\times$ & $\times$ & $\checkmark$ \\ 
\hline 
$\rL_{6,12}\oplus\rA_1$ & $\times$ & $\times$ & $\checkmark$ \\ 
\hline
$\rL_{6,13}\oplus\rA_1$ & $\times$ & $\checkmark$ & $\checkmark$ \\
\hline 
$\rL_{6,14}\oplus\rA_1$ & $\times$ & $\checkmark$ & $\checkmark$ \\ 
\hline 
$\rL_{6,15}\oplus\rA_1$ & $\times$ & $\checkmark$ & $\checkmark$ \\ 
\hline 
$\rL_{6,16}\oplus\rA_1$ & $\times$ & $\checkmark$ & $\checkmark$ \\ 
\hline
$\rL_{6,17}^+\oplus\rA_1$ & $\times$ & $\checkmark$ & $\checkmark$ \\ 
\hline
$\rL_{6,17}^-\oplus\rA_1$ & $\times$ & $\checkmark$ & $\checkmark$ \\ 
\hline
$\rL_{6,18}\oplus\rA_1$ & $\times$ & $\times$ & $\checkmark$ \\ 
\hline
$\rL_{6,19}\oplus\rA_1$ & $\times$ & $\times$ & $\checkmark$ \\ 
\hline
$\rL_{6,20}\oplus\rA_1$ & $\times$ & $\times$ & $\checkmark$ \\ 
\hline
$\rL_{6,21}\oplus\rA_1$ & $\times$ & $\checkmark$ & $\checkmark$ \\ 
\hline
$\rL_{6,22}\oplus\rA_1$ & $\times$ & $\times$ & $\checkmark$ \\ 
\hline
\end{tabu} }
\end{center}

\bibliographystyle{plain}
\bibliography{bibliography}

\end{document}